\documentclass[journal]{IEEEtran}
\usepackage{amsmath}
\usepackage{amsthm}
\usepackage{amsfonts}
\usepackage{bbm} %For some reason amsfonts doesn't let me compile \mathbbm{1}.
\usepackage{cite}

\newtheorem{defn}{Definition}
\newtheorem{thm}{Theorem}
\newtheorem{lem}{Lemma}
\newtheorem{prop}{Proposition}
\newtheorem{cor}{Corollary}

\newtheorem{rem}{Remark}
\newtheorem{exam}{Example}
\newtheorem{assump}{Assumption}
\newtheorem{prob}{Problem}

 \global\long\def\Id{\mathrm{Id}}
 \global\long\def\G{\mathcal{G}}
 \global\long\def\E{\mathcal{E}}
 \global\long\def\V{\mathbb{V}}
 \global\long\def\EE{\mathbb{E}}
 \global\long\def\R{\mathbb{R}}
 \global\long\def\IM{\mathrm{IM}}
 \global\long\def\ker{\mathrm{Ker}}
 \global\long\def\Proj{\mathrm{Proj}}

\usepackage[color=yellow, textsize=small, textwidth=\marginparwidth, draft]{todonotes}   % notes showed
\usepackage{xcolor}
\definecolor{forestgreen}{rgb}{0.13, 0.55, 0.13}
\definecolor{orange}{rgb}{1,0.49,0}

\begin{document}

\title{Analysis and Synthesis of MIMO Multi-Agent Systems Using Network Optimization}
\author{Miel Sharf and Daniel Zelazo
\thanks{M. Sharf and D. Zelazo are with the Faculty of Aerospace Engineering, Technion-Israel Institute of Technology, Haifa, Israel. This work was supported by the German-Israeli Foundation for Scientific Research and Development. 
    {\tt\small msharf@tx.technion.ac.il, dzelazo@technion.ac.il}}
}%\author

\maketitle
\begin{abstract}
This work studies analysis and synthesis problems for diffusively coupled multi-agent systems.  We focus on networks comprised of multi-input multi-output nonlinear systems that posses a property we term \emph{maximal equilibrium-independent cyclically monotone passivity} (MEICMP), which is an extension of recent passivity results.  \if(0) Following the same thread as \cite{SISO_Paper} which focused only on SISO systems, w\fi We demonstrate that networks comprised of MEICMP systems are related to a pair of dual network optimization problems.  In particular, we show that the steady-state behavior of the multi-agent system correspond to the minimizers of appropriately defined network optimization problems.  This optimization perspective leads to a synthesis procedure for designing the network controllers to achieve a desired output.  We provide detailed examples of networked systems satisfying these properties and demonstrate the results for a network of damped planar oscillators.%Exploiting this connection between the dynamic networked system and static optimization problems, we propose a synthesis procedure for designing the coupling controllers in the network to achieve a desired output state for the network.  We provide detailed examples of dynamical networked systems satisfying these properties and demonstrate the results for a network of damped planar oscillators.
\end{abstract}

%%%%%%%%%%%%%%%%%%%%%%%%%%%%%%%%%%%%%
\section{Introduction and Motivation}
%%%%%%%%%%%%%%%%%%%%%%%%%%%%%%%%%%%%%

%\textcolor{forestgreen}{Reviewer1}, \textcolor{violet}{Reviewer2}, \textcolor{orange}{Reviewer3}

Multi-agent systems have been extensively studied in recent years, mainly due to their applications in various fields in the sciences and engineering, e.g. robotics, neural networks, and power grids \cite{Robotics_Example,Neural_Network_Example,Power_Grid_Example}. The study of graphs and their algebraic representations have emerged as an important tool in the modeling and analysis of these systems \cite{Mesbahi_Egerstedt}. When the agents are considered as dynamical systems, then the notion of passivity theory brings a powerful framework to analyze the dynamic behavior of these interconnected systems.   
%
%A unified mathematical theory has been pursued by several researchers, and a key notion in such pursuits is the notion of passivity. 
Passivity theory enables an analysis of the networked system that decouples the dynamics of the agents in the ensemble, the structure of the information exchange network, and the protocols used to couple interacting agents \textcolor{black}{\cite{Arcak_Book}}. % a decoupling of the dynamics, the communication protocols and the underlying graph structure. 
Passivity for multi-agent systems was first pursued in \cite{Arcak}, where it was used to study group coordination problems. Several variants of passivity theory were used in various contexts like coordinated control of robotic systems \cite{Chopra_Spong}, synchronization problems \cite{Scardovi_Arcak_Sontag,Stan_Sepulchre}, port-Hamiltonian systems on graphs \cite{Port_Hamiltonian_Systems_On_Graphs}, and distributed optimization \cite{Tang_Hong_Yi}. %, and clustering in networks with saturated couplings \cite{Burger_Zelazo_Allgower3}.

One important variant of passivity particularly useful for the analysis of multi-agent systems is \emph{equilibrium-independent passivity} (EIP), introduced in \cite{Linear_System_Transfer_Function}.  For EIP systems, passivity is verified with respect to any steady-state input-output pair, allowing one to show convergence results without specifying the limit beforehand \textcolor{black}{\cite{SimpsonPorco}}. A generalization of this result known as \emph{maximal monotone equilibrium-independent passivity} (MEIP) was introduced in \cite{SISO_Paper} for SISO systems, allowing to prove convergence using energy methods for a much wider class of systems, including, for example, integrators, that are not EIP.  

These passivity extensions have proved very powerful in the analysis of networked systems.  Of interest to this work, \cite{SISO_Paper} used the MEIP notion to establish an equivalence between the steady-state behavior of networked systems and the solution of a pair of dual \emph{network optimization} problems.  Thus, the analysis of networks comprised of MEIP passive agents and controllers could be accomplished by studying an associated static optimization problem.  These results were recently extended in \cite{LCSS_Paper} to provide a synthesis procedure for the interaction protocols between agents to enforce a desired relative state configuration between the agents. In \cite{Jain2018a}, an equivalence between feedback passivation of passivity short systems was made with a regularization of the corresponding network optimization problems. %Both the results in \cite{SISO_Paper} and \cite{LCSS_Paper} considered networks of single-input single-out systems.

% network optimization problems and passivity-based cooperative control problems for SISO agents.  In \cite{LCSS_Paper}, this framework was used together with sub-gradient theory and convex optimization theory to solve the synthesis problem for relative outputs and SISO agents. test...

The main shortcoming of these works building upon MEIP systems is that they are only applicable to single-input single-output (SISO) agents; this is a consequence of the monotonicity requirement of the steady-state input-output relations that cannot be easily generalized to systems with more than one input or output. This shortcoming motivates the current work, where we aim to extend the notion of MEIP passivity for both the analysis and synthesis of networked systems comprised of multi-input multi-output (MIMO) systems.  

The main analytic tool required to study MIMO systems in this context is the notion of \emph{cyclically monotone} (CM) relations, originally introduced in \cite{Rockafeller}.  Cyclically monotone relations provide the correct generalization of monotonicity of scalar functions to vector functions.  The key result due to \cite{Rockafeller} shows that CM relations are contained in the sub-gradient of a convex function.  With this tool in hand, we are able to extend the SISO results from \cite{SISO_Paper} to square MIMO systems.    
The main contributions of this work can now be stated as follows:
\begin{itemize}

\item \textcolor{black}{We develop a MIMO counterpart of monotone equilibrium independent passivity using the notion of \emph{cyclically monotone relations}, which we term \emph{maximal equilibrium-independent cyclically monotone passivity} (MEICMP).}

%
%We introduce the notion of \emph{cyclically monotone relations} as a  generalization of monotonicity from scalars to vectors.  This is then used to develop a MIMO-counterpart of monotone equilibrium-independent passivity, which we term \emph{maximal equilibrium-independent cyclically monotone passivity} (MEICMP).  
%
%We extend the notion of MEIP for \emph{square} MIMO systems.  The main analytic tool is to 
%required is t \emph{maximal equilibrium-independent cyclically monotone passivity}
%We introduce the notion of \emph{maximal equilibrium-independent cyclically monotone passivity}, which is a MIMO-counterpart of monotone equilibrium-independent passivity.
%\item We use this new notion to solve the output-agreement problem, giving a necessary condition for the existence of output-agreement.
\item We show that a diffusively coupled network comprised of MIMO systems that are (output-strictly) MEICMP with  controllers that are also MEICMP converge to an output agreement steady-state. % output agreement configuration.  
Moreover, we show that the steady-states of the system are the optimal solutions of a dual pair of network optimization problems. 
 
%solve the analysis problem for MIMO multi-agent systems, while assuming maximal equilibrium-independent cyclically monotone passivity. Namely, we prove that the system converges and that the limit is a minimizer of some network optimization problem.
\item We propose a synthesis procedure for designing network controllers assuring global asymptotic convergence to a desired output.  We present conditions for when such a synthesis is feasible.  If a synthesis is not feasible, we present a practically-justifiable method for plant augmentation assuring the synthesis problem can be solved for any desired steady-state output.
%We solve the synthesis problem for MIMO multi-agent systems under this assumption, finding a criteria assuring a desired output can be forced, and finding a synthesis for controllers assuring global asymptotic convergence.
%\item We exhibit a practically-justifiable method for plant augmentation to assure that the synthesis problem can be solved for any desired steady-state output, even if it fails the criteria found.
%\item We exhibit a wide class of systems and prove that it consists of maximal equilibrium-independent cyclically monotone passive systems.
%\item Lastly, we demonstrate the results by a simulation.
\end{itemize}
%
%In addition to the fundamental analysis and synthesis results presented here, 
\noindent We also provide numerous examples of systems that can be classified as MEICMP including convex-gradient dynamical systems with oscillatory terms and damped oscillators.  %Numerical simulations are also provided.

The remainder of the paper is organized as follows. 
Section \ref{sec.CoopOverDiffusive} reviews the network model and results on passivity in multi-agent coordination.  In Section \ref{sec.CM_And_Use}, we present the main results in three subsections. \textcolor{black}{The first studies steady-states of the closed loop system. The second introduces the notion of cyclically monotone relations and shows how they can be understood via network optimization problems. %It also presents MEICMP systems as system that satisfy both passivity and cyclic monotonicity. 
The third shows that if the agents and the controllers are all assumed to be MEICMP, then the closed-loop system converges to the steady-states}. The remaining subsections provide analysis results for networks comprised of MEICMP systems and their relation to a class of dual network optimization problems. Section \ref{sec:Synthesis} deals with the synthesis problem, providing a characterization for solvable cases and a corresponding synthesis procedure.
Section \ref{sec.ConvGradSystems} presents examples of systems with cyclically monotone input-output relations. % along with some numerical simulations.

%deals with background material and previous results, and consists of two subsections. The first surveys diffusively-coupled networks, passivity and monotone relations, while the second exhibits previous results relying on the notion of monotone relations. 
%Section \ref{sec.CM_And_Use} is two-phased, as the first three subsections study the notion of cyclically monotone relations and its relation to systems theory, while the last two prove an inverse-optimality result in the spirit of \cite{SISO_Paper} using the notion of cyclically-monotone relations. 
%Section \ref{sec:Synthesis} deals with the synthesis problem, giving a characterization for solvable cases, solving the problem in these cases and giving a method to make any case solvable.
%Section \ref{sec.ConvGradSystems} deals with examples of systems with cyclically monotone input-output relations, namely discussing about gradient systems with oscillatory terms. 
%Section \ref{sec.Simulation} exemplifies the results of the paper by solving the synthesis problem for a network of MIMO damped linearly-forced oscillators. 

\noindent\paragraph*{Preliminaries} 
A graph $\mathcal{G}=(\mathbb{V},\mathbb{E})$ consists of a node set $\mathbb{V}$ and edge set $\mathbb{E}$. Each edge $k$ consists of two vertices $i,j\in\mathbb{V}$, and we orient it arbitrarily, say from $i$ to $j$; we write $k=(i,j)$ in this case. We define the incidence matrix $E$ of $\mathcal{G}$ as a $|\mathbb{V}|\times|\mathbb{E}|$ matrix such that for any edge $k=(i,j)$, $E_{ik}=-1,\ E_{jk}=1$ and all other entries in $k$'s column are zero. Given some integer $d$, \textcolor{black}{thought of as the input- and output- dimensions of the dynamical system of each of the agents, }we define the incidence operator as $\mathcal{E}=E\otimes \Id_d$, where \textcolor{black}{$\Id_d$ is the identity operator $\mathbb{R}^d\to\mathbb{R}^d$} and $\otimes$ is the Kronecker product. It is important to note that the null-space of $\E$ consists of all vectors of the form $[u^T,u^T,\ldots,u^T]^T$. 
Also, if $f:\mathbb{R}^r\to\mathbb{R}$ is a convex function, its Legendre transform is defined as $f^\star:\mathbb{R}^r\to\mathbb{R}$ by $f^\star(y)=-\inf_u \{f(u) - y^Tu\}$ \cite{Convex_Optimization}.
Furthermore, consider the indicator function for the set $\mathcal{C}$, denoted $I_{\mathcal{C}}$, defined by $I_{\mathcal{C}}(x) = 0$ whenever $x \in \mathcal{C}$, and $I_{\mathcal{C}}(x) = \infty$ otherwise.
%$$I_{\mathcal{C}}(x)=\begin{cases} 0, & x \in \mathcal{C} \\ \infty & \text{otherwise}.\end{cases}$$
Of particular interest is when the set $\mathcal{C}=\{0\}$, and we denote the indicator function in that case as $I_0(x)$.
For a map $T$, we denote its image by $\IM(T)$. If $T$ is a linear map, the kernel of $T$ is denoted by $\ker (T)$.
  From now on, we will use italicized letters (e.g., $\mathit{y_i(t)}$ or $\mathit{y_i}$) to denote time-dependent signals, and normal font letters (e.g. $\mathrm{y_i}$) to denote constant vectors.
\textcolor{black}{Finally, for sets $A,B \subset\R^N$, we define $A+B=\{x\in \R^N|x=a+b,a\in A, b\in B\}$.} % is the set consisting of all sums $a+b$ where $a\in A,\ b\in B$.} 

\textcolor{black}{
\section{Diffusively Coupled Networks and the\\ Role of Passivity in Cooperative Control} \label{sec.CoopOverDiffusive}
In this section, we introduce the network dynamic model used in this work and present an overview of the role of passivity in cooperative control.
\subsection{The Diffusively Coupled Network Model} 
}We consider a population of agents that interact over a network, described by the graph $\mathcal{G} = (\mathbb{V},\mathbb{E})$. The agents are represented by the vertices $\mathbb{V}$, and pairs of interacting agents are represented by edges $\mathbb{E}$. Each specific edge contains information about the coupling (i.e., the controllers), which are allowed to be dynamic. We assume a \emph{diffusive coupling} structure, where the input to the edge controllers are the difference between the output of the adjacent agents, and the control input of each agent is the (directed) sum of the edge controller outputs.

Each agent in the network is modeled as a square multiple-input multiple-output dynamic system of the form
\begin{align} \label{eq:NodeSystemsODE}
\Sigma_i : 
\begin{cases}
\dot{x}_i(t) = f_i(x_i(t),u_i(t),\mathrm{w}_i) \\
y_i(t) = h_i(x_i(t),u_i(t),\mathrm{w}_i) 
\end{cases}
i\in\mathbb{V},
\end{align}
where $x_i(t) \in \mathbb{R}^{p_i}$ is the state, $u_i(t)\in\mathbb{R}^d$ the input, $y_i(t)\in\mathbb{R}^d$ the output, and $\mathrm{w}_i$ a constant exogenous input. Note that each agent need not have the same state dimension, but we require all agents have the same number of inputs and outputs ($d$).  Let $u(t)=[u_1(t)^T,\dots,u_{|\mathbb{V}|}(t)^T]^T$ and $y(t)=[y_1(t)^T,\ldots,y_{|\mathbb{V}|}(t)^T]^T$ be the concatenation of the input and output vectors. Similarly, $x(t)\in\mathbb{R}^{\sum_{i=1}^{|\mathbb{V}|}p_i}$ is the stacked state vector, and $\mathrm{w}$ the stacked exogenous input.

\begin{figure} [!t] 
    \centering
    \includegraphics[scale=0.5]{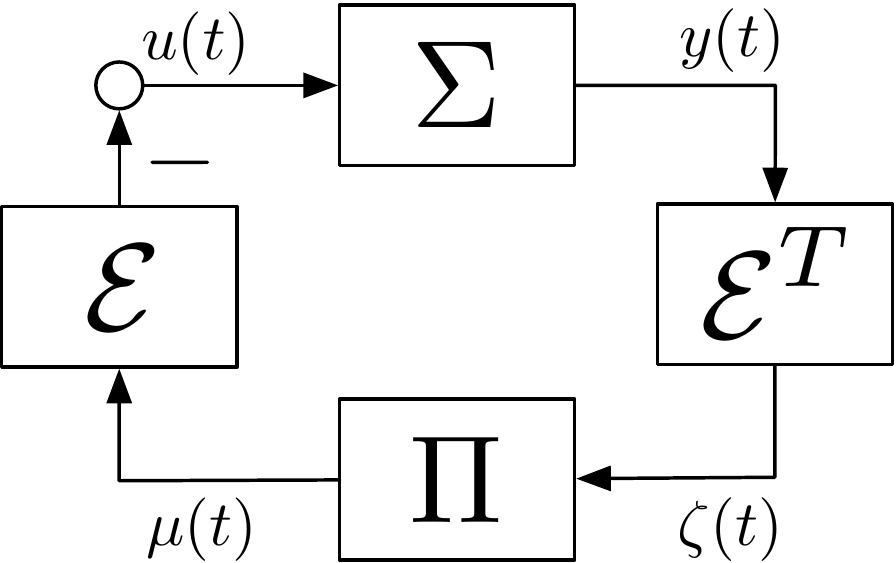}
    \caption{Block-diagram of the network system $(\Sigma, \Pi, \mathcal{G})$.}
    \label{ClosedLoop}
    \vspace{-15pt}
\end{figure}

The agents are diffusively coupled over the network also by a dynamic system that we consider as the network controllers.  For the edge $e=(i,j)$, we denote the difference between the output of the adjacent nodes as $\zeta_e(t) = y_j(t) - y_i(t)$.  The stacked vector $\zeta(t)$ can be compactly expressed using the incidence operator of the graph as $\zeta(t)=\mathcal{E}^Ty(t)$.
%
%Now, for the coupling. If we take some edge $k=(i,j)$, then we consider the difference along $k$, namely $\zeta_k(t) = y_j(t) - y_i(t)$. These can be stacked to get a vector $\zeta$ satisfying $\zeta(t)=\mathcal{E}^Ty(t)$. 
These, in turn, drive the edge controllers described by the dynamics
\begin{equation} \label{eq:EdgeSystemsODE}
\Pi_e : 
\begin{cases}
\dot{\eta}_e(t) = \phi_e(\eta_e(t),\zeta_e(t)), \\
\mu_e(t) = \psi_e(\eta_e(t),\zeta_e(t))
\end{cases}
 e\in\mathbb{E}.
\end{equation}
 % lies in the range of $\mathcal{E}^T$. %, meaning that $\eta(t)$ lies there for all times $t$. 
%The conditions on the nonlinearities $\psi_k$ will be specified later. 
The output of these controllers will yield an input to the node dynamical systems as $u(t)=-\mathcal{E}\mu(t)$, with $\mu(t)$ the stacked vector of controller outputs. We denote the complete network system by the triple $(\Sigma, \Pi, \mathcal{G})$, where $\Sigma$ and $\Pi$ are the parallel interconnection of the agent and controller systems, and $\mathcal{G}$ the underlying network; see Figure \ref{ClosedLoop}.  %A block diagram of this structure is given in Figure \ref{ClosedLoop}. %Examples for systems governed by this model can be found in \cite{Port_Hamiltonian_Systems_On_Graphs,SISO_Paper}.

\textcolor{black}{
\subsection {The Role of Passivity in Cooperative Control}
Passivity theory has taken an outstanding role in the analysis of cooperative control systems, and in particular those with the diffusive coupling structure of Figure \ref{ClosedLoop}. We dedicate this section to consider a few variants of passivity used to prove various analysis results for multi-agent systems. The main advantage of using passivity theory is that it allows to decouple the system into three different layers - namely the agent dynamics, the coupling dynamics, and the network connecting the two. This concept is clearly seen in the following theorem:}
%The main convergence result for such networks is given below.
%
\begin{thm}[\cite{SISO_Paper}] \label{ConvergencePassivity}
Consider the network system $(\Sigma,\Pi,\mathcal{G})$ comprised of agents and controllers. Suppose that there are constant vectors $\mathrm{u_i,y_i,\zeta_e}$ and $\mathrm{\mu_e}$ such that
\begin{itemize}
\item[i)] the systems $\Sigma_i$ are output strictly-passive with respect to $\mathrm{u_i}$ and $\mathrm{y_i}$;
\item[ii)] the systems $\Pi_e$ are passive with respect to $\mathrm{\zeta_e}$ and $\mathrm{\mu_e}$;
\item[iii)] the stacked vectors $\mathrm{u,y,\zeta}$ and $\mathrm{\mu}$ satisfy $\mathrm{u}=-\mathcal{E}\mathrm{\mu}$ and $\mathrm{\zeta} = \mathcal{E}^T\mathrm{y}$.
\end{itemize}
Then the output vector $\mathit{y}(t)$ converges to $\mathrm{y}$ as $t\to\infty$.
\end{thm}
\noindent\textcolor{black}{Indeed, the first condition involves the agent dynanics, the second the controllers, and the third the underlying network}.

%The importance of this theorem is that it allows us to decouple the node systems and the edge controllers. If each of them is shown to be passive, then the whole network converges to a constant steady-state value. This is the main leverage point that passivity theory enables - an independent analysis of the network structure, the couplings, and the agent dynamics.

\textcolor{black}{The first paper to fully embrace passivity theory to analyse cooperative control problems was \cite{Arcak}. This led to many variants of passivity in the literature proven to be useful for the analysis of cooperative control problems.} %began a trend of using variants of passivity to prove further analysis results in cooperative control.}
\textcolor{black}{
\emph{Incremental passivity} (IP), introduced in \cite{IncrementalPassivity}, allows one to consider the passivity property with respect to certain trajectories, rather than fixed equilibria.  
%The next variant of passivity to be used is \emph{Incremental Passivity}, defined below:
%\begin{defn}[\cite{IncrementalPassivity}]
%The dynamical system 
%\begin{align} \label{eq.Arbitrary}
%\begin{cases}
%\dot{x} = F(x,u) \\ 
%y = H(x,u).
%\end{cases}
%\end{align}
%is said to be incrementally passive if there exists a $C^1$ function $V=V(x_1,x_2)$ such that for any two inputs $u_1(t), u_2(t)$, corresponding trajectories $x_1(t),x_2(t)$ and outputs $y_1(t),y_2(t)$, the following inequality holds:
%\begin{align}
%\frac{d}{dt}(V(x_1(t),x_2(t)) \le (u_1(t)-u_2(t))^T(y_1(t)-y_2(t)).
%\end{align}
%\end{defn}
%Intuitively, this property allows us to replace the fixed equilibrium in the definition of passivity with some trajectory $(u(t),x(t),y(t))$ and prove similar results. 
Indeed, incremental passivity was used in \cite{Scardovi_Arcak_Sontag,Stan_Sepulchre} to prove various synchronization and analysis results in multi-agent system. However, IP is restrictive, as it demands the passivation inequality to hold for any two trajectories.
}

\textcolor{black}
{Other variants of passivity focus on the collection of all equilibria of a system.  In this direction, the notion of steady-state input-output maps is useful.  In the following, we focus on dynamical systems of the form
\begin{align} \label{eq:dynamics}
\Sigma : 
\begin{cases}
\dot{x}(t) = f(x(t),u(t)) \\
y(t) = h(x(t),u(t)) 
\end{cases}.
\end{align}
%
%
%like \emph{equilibrium-independent passivity} (EIP) and \emph{maximal monotone equilibrium-independent passivity} (MEIP), study the collection of all equilibria of the system. Both EIP and MEIP utilize the notion of steady-state input-output maps, defined below.
\begin{defn}
Consider the dynamical system \eqref{eq:dynamics} with input $u \in \mathcal{U}$ and output $y \in \mathcal{Y}$. The \emph{steady-state input-output set} associated with \eqref{eq:dynamics} is the set $k_y \subset \mathcal{U}\times \mathcal{Y}$ consisting of steady-state input-output pairs $(\mathrm u,\mathrm y)$ of the system.  
\end{defn}
}

\textcolor{black}
{With this definition, we now introduce the next variant of passivity termed \emph{equilibrium-independent passivity} (EIP) \cite{Linear_System_Transfer_Function}. A key feature of EIP is the assumption that for any steady-state input $\mathrm{u}$ there is exactly one steady-state output $\mathrm{y}$.  This implies that the steady-state output $\mathrm y$ can be expressed as a continuous function of the steady-state input $\mathrm u$.  Thus, with a slight abuse of notation we can consider the the set $k_y$ as a \emph{function}  $k_y:\mathrm u \mapsto \mathrm y$, i.e. $\mathrm y = k_y(\mathrm u)$. % can be thought of as a function.
%\begin{defn} [\cite{Linear_System_Transfer_Function}]
%The system \ref{eq.Arbitrary} is said to be Equilibrium-Independent Passive (EIP) if:
%\begin{itemize}
%\item[i)] The system is passive with respect to any steady-state pair $(\mathrm u,\mathrm y)$ it has.
%\item[ii)] For any steady-state input $\mathrm u$ there is exactly one steady-state output $\mathrm y$ such that $(\mathrm u,\mathrm y)$ is a steady-state input-output pair.
%\end{itemize}
%\end{defn}
%The second assumption implies that the steady-state output $\mathrm y$ can be considered as a continuous function of the steady-state input $\mathrm u$. 
In general, this is less restrictive than IP, and allows to prove analysis results for MIMO systems. However, there are IP systems which are not EIP. The epitome of these kind of systems is the simple integrator, which can be verified to be IP, but not EIP. The steady-state input $\mathrm u=0$ has multiple different steady-state outputs (depending on the initial condition of the system), and thus the input-output map is no longer a function.}

\textcolor{black}
{The last variant of passivity we review is \emph{maximal equilibrium-independent passivity} (MEIP) \cite{SISO_Paper}. It is a variant of EIP that attempts to remedy the exclusiveness of the simple integrator and similar systems. However, it is only defined in the case of SISO systems, as it relies on the notion of monotone relations:
\begin{defn}[\cite{SISO_Paper}]
Consider a relation $R\subseteq\R\times \R$. We say that $R$ is a monotone relation if for every two elements $(\mathrm u_1,\mathrm y_1)$ and $(\mathrm u_2,\mathrm y_2)$, we have that $(\mathrm u_2 - \mathrm u_1)(\mathrm y_2 - \mathrm y_1) \ge 0$. We say that $R$ is maximally monotone if it is monotone and is not contained in a larger monotone relation.
\end{defn}
In other words, increasing the first element $\mathrm u$ implies that the second element $\mathrm y$ cannot decrease. We now present the definition of MEIP.
\begin{defn} [\cite{SISO_Paper}]
The SISO system \eqref{eq:dynamics} is said to be \emph{maximal equilibrium-independent passive} (MEIP) if:
\begin{itemize}
\item[i)] The system is passive with respect to any steady-state pair $(\mathrm u,\mathrm y)$ it has.
\item[ii)] The collection $k_y$ of all steady-state input-output pairs $(\mathrm u,\mathrm y)$ is maximally monotone.
\end{itemize}
\end{defn}
This is indeed a generalization of EIP, as the function $k_y$ of an EIP system is monotone ascending \cite{Linear_System_Transfer_Function}. It can also be shown that the simple SISO integrator is MEIP. However, the problem of finding a MIMO analogue of MEIP, or a variant of EIP that will include marginally-stable systems like the simple integrator, has not been addressed in the literature.}

\textcolor{black}{
In the following section, we present a generalization of MEIP and of EIP to MIMO systems which include integrators and other marginally-stable systems. The key notion that we'll use is one possible generalization of monotonic relations from subsets of $\R\times \R$ to subsets of $\R^d\times \R^d$. This generalization is called cyclic monotonicity, and was first considered in \cite{Rockafeller}.
}

%\subsubsection{Maximally Monotone Relations}
\if(0)
Towards a more general theory of these diffusively coupled network systems, there have been many extensions to the result in Theorem \ref{ConvergencePassivity} that rely on certain refinements of passivity theory \cite{Linear_System_Transfer_Function,Wang_Bao,Li_Zhao,Summers_Arack_Packard,Shiromoto_Manchester,Burger_DePersis}.  Of interest to this work is the notion of \emph{maximal equilibrium independent passivity} (MEIP), which explores a connection between \emph{monotone relations} and passivity; we briefly review these concepts here. % theory, and diffusively coupled networks.  We briefly review these concepts here.

\begin{defn}[Monotone Relations]
Let $R$ be a subset of $\R\times\R$. %(i.e. the dimension of the input and output is $1$). 
We say that $R$ is \emph{monotone} if for any two pairs $(u_1,y_1),(u_2,y_2)\in R$, if $u_1 \le u_2$ then $y_1 \le y_2$. 
We say that $R$ is \emph{strictly monotone} if $u_1 < y_1$ implies $u_2 < y_2$.
We say that $R$ is a \emph{maximally monotone} relation if it is monotone, and it is not contained in a larger monotone relation.
\end{defn}

It is well known that monotone relations are gradients of convex functions, %$\R\to\R$
and that these functions are unique up to an additive constant (see chapter 24 of \cite{Rockafellar1}). \textcolor{black}{We call these functions \emph{integral functions}}. We now define a system-theoretic analogue of this notion.  
\begin{defn}[MEIP \cite{SISO_Paper}]
A SISO system $\Sigma$ is called \emph{maximally monotone equilibrium-independent passive} (MEIP) if
\begin{itemize}
\item[i)] for every steady-state input-output pair $(\mathrm{u},\mathrm{y})$, the system $\Sigma$ is (output strictly) passive with respect to $\mathrm{u}$ and $\mathrm{y}$;
\item[ii)] the steady-state input-output relation, i.e., the set of all steady-state input-output pairs, is a maximally monotone relation.
\end{itemize}
\end{defn}

The main result of \cite{SISO_Paper} shows that the steady-state behavior of a network system comprised of MEIP systems converges to the solutions of certain static optimization problems.  These optimization problems belong to the class of \emph{network optimization} \cite{Rockafellar1,Bertsekas} problems, which we briefly review here.
%\subsection{Network Optimization} \label{subsec.MaxMonotoneNetworkOpt}
%Network optimization is the mathematical theory studying optimization problems on graphs \cite{Rockafellar1,Bertsekas}. It has various applications in communication systems, theoretical computer-science, and industrial engineering. Different problems have been extensively studied in this field over the years - shortest path problems, cheapest spanning-tree problems, and many more, but the two most important problems are the \textit{optimal potential} and \textit{optimal flow} problems. 

Network optimization is generally concerned with finding optimal solutions to problems defined over network structures.  This includes the celebrated shortest path problems and max-flow/min-cut problems.  The basic structure of a network optimization problem includes a graph with optimization variables and cost functions associated to the nodes and edges, along with constraints involving the structure of the graph.  

The basic terminology in network optimization is now presented.  A \emph{flow} on a network is a vector $\mu = [\mu_1,\mu_2,...,\mu_{|\mathbb{E}|}]^T$. The element $\mu_e$ can be thought of as the flux through the edge $e$. The net (directed) flow through a node is the sum of the fluxes along the adjacent edges.  Thus, each node is associated with a variable $\mathrm{u}_i$ termed the \emph{divergence} of the network at the node $i$.  The relationship between flow and divergence is captured by the constraint $\mathrm{u}+{E}\mu = 0$.  This is the well known Kirchoff's current law when flows are interpreted as currents through an electrical network.
%If $u\in\mathbb{R}^{|\mathbb{V}|}$, then the $i$-th entry $\mathrm{u}_i$ can be thought of as the divergence of the network at the node $i$. In this setting, Kirchoff's current law is represented by $\mathrm{u}+{E}\mu = 0$. 

A dual interpretation for the network variables can also be made. A vector $\mathrm{y}\in\mathbb{R}^{|\mathbb{V}|}$ assigned to the nodes is termed a \emph{potential} of the network. To each edge $e=(i,j)$ one can then associate a potential difference $\zeta_e = \mathrm{y}_j - \mathrm{y}_i$, which is also called a \emph{tension}. The tension vector $\zeta\in\mathbb{R}^{|\mathbb{E}|}$ can also be expressed by $\zeta =E^T\mathrm{y}$. The flow, divergence, potential, and tension variables are all related via the incidence matrix of the graph and the corresponding \emph{conversion formula}, $\mu^T\zeta = -\mathrm{y}^T\mathrm{u}$ \cite{Rockafellar1}. %It will be convenient later to use the equality $\mu^T\zeta = \mathrm{-y^Tu}$.

The two canonical problems in network optimization are the \emph{optimal flow problem} (OFP) and the \emph{optimal potential problem} (OPP).  The OFP aims to optimize the flow and divergence in a network subject to the flow conservation constraint.  The OPP similarly optimizes the tensions and potentials in the network.  When the object functions are convex, the two problems turn out to be dual to each other. These problems are summarized below.  Here, we denote by $K(\mathrm{u})$ the divergence cost function and $\Gamma^\star(\mu)$ the flow cost function.  The function $K^\star(\cdot)$ is the convex conjugate of $K(\cdot)$, and $\Gamma^\star(\cdot)$ the convex conjugate of $\Gamma(\cdot)$.

\begin{center}
\begin{tabular}{ c||c }
 \textbf{Optimal Potential Problem}  & \textbf{Optimal Flow Problem}   \\
 (OPP) & (OFP)  \\\hline
 $ \begin{array}{cl} \underset{\mathrm{y,\zeta}}{\min} &K^\star(\mathrm{y}) + \Gamma(\zeta)\\
s.t.&{\E}^T\mathrm{y} = \zeta 
\end{array} $&  $ \begin{array}{cl}\underset{\mathrm{u,\mu}}{\min}& K(\mathrm{u}) + \Gamma^\star(\mu) \\
s.t. &\mathrm{u} = -{\E}\mu.
\end{array} $ 
\end{tabular}
\end{center}

%One can consider a few optimization problems over networks, each will be given an interpretation in the cooperative control setting. The first attempts to optimize the flow and divergence in a network, subject to a conservation of flow constraint. We present a flux tariff, giving a cost $C_k^{flux}(\mu_k)$ to a flux of volume $\mu_k$, and a divergence tariff, giving a cost $C_i^{div}(u_i)$ to each divergence $u_i$. In this case, one tries to minimize,
%\begin{align}
%\min_{u,\mu}&\quad \sum_{i=1}^{|\mathbb{V}|} C_i^{div}(u_i) + \sum_{k=1}^{|\mathbb{E}|} C_k^{flux}(\mu_k) \\
%s.t.& \quad u+\E\mu = 0 \nonumber.
%\end{align}

%If the tariffs are convex functions, one can consider the dual problem, in the convex optimization sense. This is done by defining $C^{pot}_i = (C_i^{div})^\star$ and $C_k^{ten} = (C_k^{flux})^\star$, yielding the dual problem, known as the\emph{optimal potential problem}:
%\begin{align}
%\min_{y,\zeta} & \quad \sum_{i=1}^{|\mathbb{V}|} C_i^{pot}(y_i) + \sum_{k=1}^{|\mathbb{E}|} C_k^{ten}(\mu_k) \\
%s.t. & \quad \zeta=\mathcal{E}^Ty \nonumber ,
%\end{align}
%in which the goal is to minimize the cost over a network where each node has a cost function for potentials, and each edge has a cost function for tensions.

We are now prepared to present the main result from \cite{SISO_Paper}.

\begin{thm}[\cite{SISO_Paper}]\label{thm.SISO}
Consider the network system $(\Sigma,\Pi,\mathcal{G})$ comprised of SISO agents and controllers (i.e., $d=1$).  Assume all the agent dynamics $\Sigma_i$ for $i=1,\ldots,|\mathbb{V}|$ are (output-strictly) MEIP, and that all couplings $\Pi_e$ for $e=1,\ldots,|\mathbb{E}|$ are MEIP.
Let $K_i$ and $\Gamma_e$ be integral functions of the corresponding input-output relations. Then the signals $u(t),y(t),\mu(t),\zeta(t)$ converge to constant signals $\mathrm{u,y,\mu,\zeta}$, which are minimizers of the optimal potential and optimal flow problems.
%Suppose that all of the agents $\Sigma_i$ in Figure \ref{ClosedLoop} are (output-strictly) MEIP, and that all couplings $\Pi_e$ are MEIP. Let $K_i$ and $\Gamma_e$ be integral functions of the corresponding input-output relations. Then the signals $u(t),y(t),\mu(t),\zeta(t)$ converge to constant signals $\mathrm{u,y,\mu,\zeta}$, which are minimizers of the following optimal potential and optimal flow problems:

Theorem \ref{thm.SISO} thus shows that it is of interest to study the optimal solutions to (OPP) and (OFP) when analyzing these network systems.  In this direction, we present a useful lemma showing an optimality condition for (OPP) and (OFP) that will be utilized in the sequel.
%The goal of proving that certain signals are optimal solutions to (OPP) and (OFP) now seems obvious. The main tool that will be used is the following lemma.

\begin{lem} \label{lem:ZeroEquationLemma}
Consider the problems (OPP) and (OFP) and assume that $K$ and $\Gamma$ are any convex functions.
Let $p_0=(\mathrm{u_0},\mu_0)$ and $d_0=(\mathrm{y_0},\zeta_0)$ be feasible solutions to ({OFP}) and ({OPP}) respectively. If $K(\mathrm{u_0}) + K^\star(\mathrm{y_0}) + \Gamma(\zeta_0)+\Gamma^\star(\mu_0)= 0$, then $p_0$ and $d_0$ are dual optimal solutions.
\end{lem}
\begin{proof} 
We first note that if $(\mathrm{u},\mu)$ and $(\mathrm{y},\zeta)$ are any feasible solutions, we have that
\[
\mathrm{u}^T\mathrm{y}+\mu^T\zeta=-\mu^T\E^T\mathrm{y}+ \mu^T\E^T\mathrm{y}=0
\]
Now, we define the function $F(u,\mu)=K(u)+\Gamma^\star(\mu)$ which has a dual $F^\star(y,\zeta)=K^\star(y)+\Gamma(\zeta)$. Note that the assumption is that $F(\mathrm{u_0},\mu_0)+F^\star(\mathrm{y_0},\mu_0) = 0$. By the definition of the Legendre transform, if $(\mathrm{u},\mu)$ and $(\mathrm{y},\zeta)$  are again any feasible solutions, then we have 
\begin{align}\label{eq:LegendreDualityIneq}
F^\star(\mathrm y,\zeta) &= \sup_{(\mathrm{\hat{u}},\hat\mu)} \{{(\mathrm{\hat u},\hat\mu)^T(\mathrm y,\zeta)-F(\mathrm{\hat u},\hat\mu)}\} \\
&\ge (\mathrm{u},\mu)^T(\mathrm y,\zeta)-F(\mathrm{u},\mu) = -F(\mathrm{u},\mu). \nonumber
\end{align}
%Thus, abusing \eqref{eq:LegendreDualityIneq} yields that 
In particular, we have $F^\star(\mathrm y,\zeta) \ge - F(\mathrm u_0, \mu_0) = F^\star(\mathrm y_0,\zeta_0)$, and similarly $F(\mathrm u,\mu) \ge - F^\star(\mathrm y_0, \zeta_0) = F(\mathrm u_0,\mu_0).$
The duality follows from \eqref{eq:LegendreDualityIneq} and $F(\mathrm{u_0},\mu_0)+F^\star(\mathrm{y_0},\mu_0) = 0$, meaning that the supremum appearing in the definition of $F^\star(\mathrm y_0,\zeta_0)$ is achieved at $(\mathrm u_0,\mu_0)$.
\end{proof}

\begin{cor}\label{cor.feasopt}
If $\mathrm{u_0,y_0}$ are feasible solutions to (OPP) and (OFP) with $\Gamma(\zeta)=I_0(\zeta)$ and they satisfy $K(\mathrm u_0) + K^\star(\mathrm y_0) = 0$, then they are dual optimal solutions.
\end{cor}
\fi

\section{Cyclically Monotone Relations\\ and Cooperative Control} \label{sec.CM_And_Use}

In \cite{SISO_Paper}, the concept of monotone relations is used to provide convergence results for a network system $(\Sigma, \Pi, \mathcal{G})$ comprised of SISO agents. However, many applications deal with MIMO systems, necessitating a need to extend this work for network systems consisting of MIMO agents. {\color{black}{We begin by considering the steady-states of the system.}} %From now on, fix the plants $\Sigma_i$, the underlying graph $\G$, and controllers $\Pi_e$.
{\color{black}{
\subsection{Steady-States and Network Consistency}
Consider a steady-state $(\mathrm{u,y,\zeta,\mu})$ of the closed loop system in Fig. \ref{ClosedLoop}. We know that for every $i=1,...,|\V|$, $(\mathrm{u}_i,\mathrm{y}_i)$ is a steady-state input-output pair of the $i$-th agent $\Sigma_i$. Similarly, for every $e\in \EE$, $(\mathrm{\zeta}_e,\mathrm{\mu}_e)$ is a steady-state input-output pair of the $e$-th controller $\Pi_e$.  The network interconnection between the systems $\Sigma$ and $\Pi$ imposes an additional consistency constraint between these steady-state values. %, but we should also require that these should be consistent with one another, i.e. that $\E$ links them together accordingly. 
This motivates us to consider the steady-state input-output relations for each of the agents and the controllers. 

In this direction, we denote the steady-state input-output relation of the $i$-th agent by $k_i$, and the relation for the $e$-th controller by $\gamma_e$. That is, $k_i \subset \mathbb{R}^d\times\mathbb{R}^d$ and $\gamma_e \subset \mathbb{R}^d\times\mathbb{R}^d$. %\textcolor{red}{(is this really the first place you introduce this notation? are the steady-state maps used also in Section II?)}\textcolor{yellow}{It is. I discuss input-output relations, but only to define MEIP} 
We denote the stacked relation for the agents and controllers as $k$ and $\gamma$, respectively. %, and the stacked relation for the controllers as $\gamma$.

\begin{rem}
Suppose that $k$ is a steady-state input-output relation. We can consider a set-valued map, also denoted by $k$, taking a steady-state input  $\mathrm u$ to the set $k(\mathrm u) = \{\mathrm y:\ (\mathrm{u,y})\in k\}$. Similarly, one can consider the inverse set-valued map taking a steady-state output $\mathrm y$ to the set $k^{-1}(\mathrm y) = \{\mathrm u:\ (\mathrm{u,y})\in k\}$.  Thus, with a slight abuse of notation we refer to $k$ ($\gamma$) as both a relation and a set-valued map.
\end{rem}

\begin{prop} \label{prop.4Requirements}
Let $\mathrm u\in \R^{d|\V|}, \mathrm y \in \R^{d|\V|},\zeta \in \R^{d|\EE|},\mu \in R^{d|\EE|}$ be any four constant vectors. Then $(\mathrm{u,y,\zeta,\mu})$ is a steady-state of the closed-loop system $(\Sigma, \Pi, \G)$ if and only if
\begin{equation}\label{eq.4Requirements}
\begin{gathered}
(\mathrm{u,y})\in k, \quad (\zeta,\mu)\in \gamma, \\
\mathrm{\zeta} = \E^T\mathrm{y}, \quad \mathrm{u} = -\E\mu.
\end{gathered}
\end{equation}
%\begin{align}\label{eq.4Requirements}
%(\mathrm{u,y})\in k , \ \mathrm{\zeta} = \E^T\mathrm{y},\ (\zeta,\mu)\in \gamma,\ \mathrm{u} = -\E\mu.
%\end{align}
\end{prop}
\begin{proof}
Follows directly from the interconnection of the network, and from the definitions of $k$ and $\gamma$.
\end{proof}
We wish to manipulate the conditions in \eqref{eq.4Requirements} to reduce the steady-state characterization from a system with four constraints to one.
%
%Now we can manipulate \eqref{eq.4Requirements} to get an easier description of the steady-state of the closed-loop network:
%Using this notation, the four conditions of \eqref{eq.4Requirements} can be restated as follows:
\begin{prop} \label{prop.1Requirement_y}
Let $\mathrm y \in \R^{d|\V|}$ be any vector. Then the following conditions are equivalent:
\begin{itemize}
\item[i)] The zero vector ${\bf 0}$ belongs to the set $k^{-1}(\mathrm y) + \E\gamma(\E^T\mathrm y)$.
\item[i)] There exists vectors $\mathrm u,\zeta,\mu$ such that $(\mathrm{u,y},\zeta,\mu)$ is a steady-state of the closed-loop network $(\Sigma, \Pi, \G)$.
\end{itemize}
\end{prop}
\begin{proof}
First, assume the existence of $\mathrm u,\zeta,\mu$. By Proposition \ref{prop.4Requirements}, it follows that
$
\mathrm u \in k^{-1}(\mathrm y),\ \zeta = \E^T\mathrm y,\ \mu\in\gamma(\zeta),\ \text{and } \mathrm u=-\E\mu.
$
Thus,
\[
{\bf 0} = \mathrm u+\E\mu \in k^{-1}(\mathrm y) + \E\gamma(\zeta) = k^{-1}(\mathrm y) + \E\gamma(\E^T\mathrm y).
\]
Conversely, if ${\bf 0}\in k^{-1}(\mathrm y) + \E\gamma(\E^T\mathrm y)$, then we know that there are some $\mathrm u\in k^{-1}(\mathrm y)$ and $\mu \in \gamma(\E^T \mathrm y)$ such that $\mathrm u+\E\mu = {\bf 0}$. Thus, by Proposition \ref{prop.4Requirements}, the 4-tuple $(\mathrm u,\mathrm y, \zeta=\E^T\mathrm y, \mu)$ is a steady-state of the closed-loop system.
\end{proof}
By the same methods, we can also reduce the conditions \eqref{eq.4Requirements} to an inclusion in the edge-variables $\mu$.
\begin{prop} \label{prop.1Requirement_mu}
Let $\mu \in \R^{d|\EE|}$ be any vector. Then the following conditions are equivalent:
\begin{itemize}
\item[i)] The zero vector ${\bf 0}$ belongs to the set $\gamma^{-1}(\mu) - \E^T k(-\E\mu)$.
\item[ii)] There exists vectors $\mathrm {u,y},\zeta$ such that $(\mathrm{u,y},\zeta,\mu)$ is a steady-state of the closed-loop network $(\Sigma, \Pi, \G)$.
\end{itemize}
\end{prop}
\begin{proof}
Same as the proof of Proposition \ref{prop.1Requirement_y}.
\end{proof}
\subsection{Connecting Steady-States to Network Optimization}
So far, we showed that the steady-states of the closed-loop system can be understood using the following two conditions:
\begin{align} \label{eq.2Requirements}
\begin{cases}
{\bf 0} \in k^{-1}(\mathrm y) + \E\gamma(\E^T \mathrm y) \\
{\bf 0} \in \gamma^{-1}(\mu) - \E^T k (-\E\mu).
\end{cases}
\end{align}
However, these conditions are highly nonlinear, and would be difficult to solve even if they were equations instead of inclusions. One method of dealing with nonlinear equations of the form $g(x) = 0$ for some function $g$, is to consider its \emph{integral function} instead. Suppose there is a function $G$ such that $g = \nabla G$. In that case, we can find a solution to $g(x) = 0$ by solving the unconstrained minimization problem $\min\,G(x)$.  If, in addition, the function $G$ is convex, the solution to the minimization problem can often be computed efficiently (i.e. in polynomial time).
%
%Such an unconstrained minimization problem can be solved numerically using gradient descent, or other extremum-seeking algorithms, as long as $G$ is convex. 

%In our case, however, we have inclusions instead of equations, but we'll try to apply similar tools. 
%
In general, convex functions need not be smooth, or even differentiable. %One example is the euclidean norm. 
In this case, the notion of the \emph{subdifferential} of a convex function can be employed. The subdifferential of the convex function $G$ at the point $x$ is denoted $\partial G(x)$, and consists of all vectors $v$ such that
\[
G(y) \ge G(x) + v^T(y-x),\ \forall y.
\]
See \cite{Rockafellar1} for more on subdifferentials. Note that the subdifferential $\partial G$ is a set-valued map. Also, analogously to the differentiable case, $x$ is a minimum point of $G$ if and only if ${\bf 0}\in\partial G(x)$. Thus, if we are able to require that $k_i$ and $\gamma_e$ are gradients of convex functions (i.e., their integral functions are convex), then the nonlinear inclusions in \eqref{eq.2Requirements} may be solved using convex optimization. 
%This hints at stating the conditions \eqref{eq.2Requirements} as convex minimization problems, and also hints that we require that $k_i$ and $\gamma_e$ are gradients of convex functions \textcolor{red}{(reword - this is backwards)}. 
In fact, these functions have been characterized due to Rockafellar \cite{Rockafeller}:}}

\begin{defn}[Cyclic Monotonicity]
Let $d\ge1$ be an integer, and consider a subset $R$ of $\mathbb{R}^d\times\mathbb{R}^d$.
We say that $R$ is a \emph{cyclically monotonic} (CM) relation if for any $N\ge1$ and any pairs $(u_{1},y_{1}),\ldots,(u_{N},y_{N})\in R$
of $d$-vectors, the following inequality holds,
\begin{equation} \label{eq:CM}
\sum_{i=1}^{N}y_{i}^{T}(u_{i}-u_{i-1})\ge 0.
\end{equation}
Here, we use the convention that $u_0=u_N$.
We say that $R$ is \emph{strictly cyclically monotonic} (SCM) if the
inequality (\ref{eq:CM}) is sharp whenever at least two $u_{i}$'s are distinct. %, 
We term the relation as \emph{maximal} CM (or maximal SCM) if it is not
strictly contained in a larger CM (SCM) relation.
\end{defn}
\textcolor{black}{This is a generalization of the concept of monotone relations for SISO system, which we elaborate upon later.} % We shall elaborate on it later in a designated section.}
\begin{thm}[\cite{Rockafeller}] \label{thm:Rockafellars} 
A relation $R\subset\mathbb{R}^n\times\mathbb{R}^n$ is the subgradient of a convex function if and only if it is maximal CM.  Moreover, it is the sub-gradient of a strictly convex function if and only if it is maximal SCM. The convex function is unique up to an additive scalar.
\end{thm}
\vspace{-15pt}
\textcolor{black}{
\begin{rem}
If $R$ is maximally CM, and $f$ is a convex function such that $R=\partial f$, then $f$ is the \emph{integral function} of $R$. 
\end{rem}}
%\textcolor{red}{(we actually use the term ``integral function" earlier - and the reviewer was confused by this...so maybe we need to define it more formally sooner?)}

\textcolor{black}{
Rockafellar's Theorem gives us a way to check that a relation is the subdifferential of a convex function. If we want to state the conditions in \eqref{eq.2Requirements} as the solutions of convex minimization problems, we need to assume that the input-output relations appearing are CM. This, together with \eqref{ConvergencePassivity}, motivates the following system-theoretic definition:}
\begin{defn}
A system $\Sigma$ is \emph{maximal equilibrium-independent cyclically monotone  (output strictly) passive} (MEICMP) if
\begin{itemize}
\item[i)] for every steady-state input-output pair $(\mathrm{u},\mathrm{y})$, the system $\Sigma$ is (output strictly) passive  with respect to  $\mathrm{u}$ and $\mathrm{y}$;
\item[ii)] the set of all steady-state input-output pairs, $R$, is maximally (strictly) cyclically monotonic.
\end{itemize}
If $R$ is strictly cyclically monotone, we say that $\Sigma$ is \emph{maximal equilibrium-independent strictly cyclically monotone} (output strictly) passive (MEISCMP).
\end{defn}
\vspace{-15pt}

\textcolor{black}{
\begin{rem} \label{rem.EveryMEIPisMEICMP}
It can be shown that when $d=1$, a relation is cyclically monotone if and only if it is monotone. Thus, a SISO system is MEIP if and only if it is MEICMP \cite{Rockafellar1}.
\end{rem}
}
\textcolor{black}{Now, suppose that the agents ${\Sigma_i}$ and the controllers ${\Pi_e}$ are all MEICMP with steady-state input maps $k_i$ and $\gamma_e$. We let $K_i$ and $\Gamma_e$ be the associated integral functions, which as a result of Theorem \ref{thm:Rockafellars}, are convex functions. %, which are often called the integral functions \textcolor{red}{(associated convex functions of what? again you are "introducing" term integral function)}. 
We let $K=\sum_i K_i$ and $\Gamma = \sum_e \Gamma_e$ be their sum, so that $\partial K = k$ and $\partial \Gamma = \gamma$. As these are convex functions, we can look at the dual convex functions $K^\star$ and $\Gamma^\star$, namely
\[
K^\star(y) = -\inf_{\mathrm u} {K(u) - y^T u},
\] and similarly for $\Gamma^\star$ \cite{Rockafellar1}. These are convex functions that satisfy $\partial K^\star = k^{-1}$ and $\partial \Gamma^\star = \gamma^{-1}$ \cite{Rockafellar1}. The functions $K,K^\star,\Gamma,\Gamma^\star$ allows us to convert the conditions \eqref{eq.2Requirements} to the unconstrained minimization problems of $K^\star(\mathrm{y})+\Gamma(\E^T\mathrm y)$ and $K(-\E\mu)+\Gamma^\star(\mu)$. Recalling that $\mathrm{u} = -\E \mu$ and that $\zeta=\E^T\mathrm y$, we can state the minimization problems in the following form: 
\begin{center}
\begin{tabular}{ c||c }
 \textbf{Optimal Potential Problem}  & \textbf{Optimal Flow Problem}   \\
 (OPP) & (OFP)  \\\hline
 $ \begin{array}{cl} \underset{\mathrm{y,\zeta}}{\min} &K^\star(\mathrm{y}) + \Gamma(\zeta)\\
s.t.&{\E}^T\mathrm{y} = \zeta 
\end{array} $&  $ \begin{array}{cl}\underset{\mathrm{u,\mu}}{\min}& K(\mathrm{u}) + \Gamma^\star(\mu) \\
s.t. &\mathrm{u} = -{\E}\mu.
\end{array} $ 
\end{tabular}
\end{center}
These static optimization problems, known as the \emph{Optimal Potential Problem} and \emph{Optimal Flow Problem}, are two fundamental problems in the field of network optimization, which has been widely studied in computer science, mathematics, and operations research for many years \cite{Rockafellar1}. A well-known instance of these problems is the \emph{maximum-flow/minimum-cut problems}, which are widely used by algorithmists and by supply chain designers \cite{CormenRivest}.} % \textcolor{red}{(is this comment motivated by a reviewer?) - (No, I thought it would be a nice touch here to say that we arrived at something which has been extensively researched)}.} 

\textcolor{black}{ We conclude this section by stating the connection between the steady-states of the closed-loop network and the network optimization problems.
\begin{thm}
Consider a network system $(\Sigma,\Pi,\G)$  and suppose that both the agents and controllers are maximally equilibrium-independent cyclically-monotone passive. Let $K$ and $\Gamma$ be the sum of the integral functions for the agents and for the controllers, respectively. For any 4-tuple of vectors $(\mathrm u,\mathrm y,\zeta,\mu)$, the following conditions are equivalent:
\begin{itemize}
\item[i)] $(\mathrm u,\mathrm y,\zeta,\mu)$ is a steady-state of the closed-loop;
\item[ii)] $(\mathrm u,\mu)$ and $(\mathrm y,\zeta)$ are dual optimal solutions of (OFP) and (OPP) respectively.
\end{itemize}
\end{thm}
\begin{proof}
We know that a convex function $F$ is minimized at a point $x$ if and only if $0\in \partial F(x)$. Applying this to the objective functions of (OPP) and (OFP) implies that they are minimized exactly when the following inclusions hold,
\begin{align} \label{eq.inclusions}
\begin{cases}
{\bf 0} \in k^{-1}(\mathrm y) + \E\gamma(\E^T \mathrm y) \\
{\bf 0} \in \gamma^{-1}(\mu) - \E^T k (-\E\mu).
\end{cases}
\end{align}
Thus, Propositions \ref{prop.1Requirement_y} and \ref{prop.1Requirement_mu} imply that if $(\mathrm {u,y},\zeta,\mu)$ is a steady-state of the closed-loop, then $(\mathrm u,\mu)$ and $(\mathrm y,\zeta)$ are optimal solutions of (OPP) and (OFP). The duality between them follows from $\mathrm y=k(\mathrm u)$, $\mu = \gamma(\zeta)$.
Conversely, if $(\mathrm u,\mu)$ and $(\mathrm y, \zeta)$ are dual optimal solutions, then $\mathrm y$ minimizes $K^\star(y)+\Gamma(\E^T y)$ and $\mu$ minimizes $K(-\E\mu)+\Gamma^\star(\mu)$. Again, a convex function is minimized only where ${\bf 0}$ is in its subdifferential, so we get the same inclusions \eqref{eq.inclusions}. By Propositions \ref{prop.1Requirement_y} and \ref{prop.1Requirement_mu} we get that $(\mathrm {u,y},\zeta,\mu)$ must be a steady-state of the closed-loop.
\end{proof}
\begin{rem}
The problems (OPP) and (OFP) are special as they are convex duals of each other; the cost functions $K^\star(\mathrm y) + \Gamma(\zeta)$ and $K(\mathrm u)+\Gamma^\star(\mu)$ are dual
\cite{Rockafellar1}. Consequently, if $(\mathrm y,\zeta)$ is an optimal solution of (OPP), then $(\mathrm u,\mu)$ is an optimal solution of (OFP) if and only if $\mu \in\gamma(\zeta),\ \mathrm u\in k^{-1}(\mathrm y)$ and $\mathrm u=-\E\mu$. Thus, solving (OPP) and (OFP) on their own gives a viable method to understand the steady-states $(\Sigma,\Pi,\G)$. %of the closed-loop network. 
\end{rem}
\subsection{Convergence to the Steady-State}
Up to now, we dealt with the steady-states of the closed-loop system, but we did not prove that the system converges to the steady-state. We now address this point. %We state the following theorem regarding convergence,
\begin{thm} \label{thm.Convergence}
Consider the network system $(\Sigma, \Pi, \G)$, and suppose all node dynamics are maximally equilibrium-independent cyclically monotone output-strictly passive and that the controller dynamics are maximally equilibrium-independent cyclically monotone passive. Then there exists constant vectors $\mathrm{u,y,\mu,\eta}$ such that $\lim_{t\rightarrow\infty} u(t) = \mathrm{u}$, $\lim_{t\rightarrow\infty} y(t) = \mathrm{y}$, $\lim_{t\rightarrow\infty} \mu(t) = \mathrm{\mu}$, and $\lim_{t\rightarrow\infty} \eta(t) = \mathrm{\eta}$. Moreover, $(\mathrm u, \zeta)$ and $(\mathrm y,\zeta)$ form optimal dual solutions to (OPP) and (OFP).
\end{thm}
We will give a proof of Theorem \ref{thm.Convergence} for the case in which the controllers are given by the following form:
\begin{align}
\Pi_e : \begin{cases} \dot{\eta_e} = \zeta_e \\ \mu_e = \psi_e (\eta_e). \end{cases}
\end{align}
The proof for the general case is analogous but more involved and is therefore not considered here. %The proof of the theorem is presented below. %\textcolor{red}{(Explain why in rebuttal - namely due to space concerns and a desire to streamline the main idea of the paper, that talks about passivity as a tool for analysis of multi-agent systems and not on minimal models for the agents)}
} %TextColor Section 3
\begin{proof}
Our assumption implies that the optimization problems (OPP) and (OFP) have dual optimal solutions solutions, meaning that a steady-state solution exists. The equilibrium-independent passivity assumption implies that there are storage functions $S_i$ (for $i\in\mathbb{V}$) and $W_e$ (for $e\in\mathbb{E}$), such that
\begin{align}
\begin{cases}
 \dot{S_i} &\hspace{-8pt}\le -\rho_i ||y_i(t)-\mathrm{y}_i||^2 + (y_i(t) - \mathrm{y}_i)^T(u_i(t) - \mathrm{u}_i)\\
\dot{W_e} &\hspace{-8pt}\le  (\mu_e(t) - \mathrm{\mu}_e)^T(\zeta_e(t) - \mathrm{\zeta}_e) \end{cases} .
\end{align}
Theorem \ref{ConvergencePassivity} implies that $y(t)$ converges to $\mathrm{y}$, implying that $\zeta(t)$ converges to $0 = \mathrm{\zeta} = \mathcal{E}^T\mathrm{y}$. Integrating implies that $\eta(t)$ converges to some $\mathrm{\eta}$, as $\dot{\eta}= \zeta$.
In turn, this implies that $\mu(t)$ converges to $\mathrm{\mu} = \psi(\mathrm{\eta})$ and that $u(t)$ converges to $\mathrm{u} = -\mathcal{E}\mathrm{\mu}$. It is clear that $(\mathrm{u,y})$ is a steady-state pair, and furthermore that $(\mathrm{u,y,\zeta,\mu})$ satisfy the conditions in Proposition \ref{prop.4Requirements}, meaning that it is also a steady-state of the closed-loop and thus it gives rise to an optimal solution of (OPP) and (OFP). This concludes the proof of the theorem.
\end{proof}

\textcolor{black}{\begin{rem}
As a consequence of Remark \ref{rem.EveryMEIPisMEICMP}, Theorem \ref{thm.Convergence} also holds for output-strictly MEIP SISO agents and MEIP SISO controllers.  This is the analysis result that was achieved in \cite{SISO_Paper}.  The result presented here is therefore more general, and the proof derivation, relying on integrating steady-state equations (or inclusions), provides a different approach than what was presented in \cite{SISO_Paper}.
%
%
%
%Because of Remark \ref{rem.EveryMEIPisMEICMP}, we get the same theorem for output-strictly MEIP SISO agents and MEIP SISO controllers. This is the analysis result that was achieved in \cite{SISO_Paper}. However, we achieved the same result through an alternative derivation, relying on integrating steady-state equations (or inclusions) rather than analyzing the network optimization problem using Lagrangians.
\end{rem}}

\if(0)
In \cite{SISO_Paper}, the concept of monotone relations is used to provide convergence results for a network system $(\Sigma, \Pi, \mathcal{G})$ comprised of SISO agents (Theorem \ref{thm.SISO}). % network convergence in the case of SISO agents. 
However, in most applications the systems in play are MIMO, especially if one considers generalized plants including uncertainties. In this direction, we aim to extend this work for network systems consisting of MIMO agents. The main analytic tool required for this extension is to generalize the notion of monotone relations to handle vector valued relations - this is accomplished by the notion of \emph{cyclically monotone relations}.  
%
%For that reason, we would like to push the envelope further and prove network convergence in the case of MIMO agents. To do so, we use the notion of \textit{cyclically monotone} relations, which was seen to generalize the notion of monotone relations. 
The section is structured as follows - first we introduce and explore the notion of cyclically monotone relations, and then show how it can be used to prove an analysis result in the spirit of \cite{SISO_Paper}. %the third connects it to systems theory, and the last two prove an analysis result in the spirit of \cite{SISO_Paper}.

\subsection{Cyclically Monotone Relations} \label{subsec.CM}
The goal of this subsection is to present the notion of cyclically monotone relations.

\begin{defn}[Cyclic Monotonicity]
Let $n\ge1$ be an integer, and consider a subset $R$ of $\mathbb{R}^n\times\mathbb{R}^n$.
%
%\begin{itemize}
%\item 
We say that $R$ is a \emph{cyclically monotonic} (CM) relation if for any $N\ge1$ and any pairs $(u_{1},y_{1}),\ldots,(u_{N},y_{N})\in R$
of $n$-vectors, the following inequality holds,
\begin{equation} \label{eq:CM}
\sum_{i=1}^{N}y_{i}^{T}(u_{i}-u_{i-1})\ge 0.
\end{equation}
%
%\item 
Here, we use the convention that $u_0=u_N$.
We say that $R$ is \emph{strictly cyclically monotonic} (SCM) if the
inequality (\ref{eq:CM}) is sharp whenever at least two $u_{i}$s are distinct. %, 
%where we use the convention that $u_{0}=u_{N}$. 
%
%\item 
We term the relation as \emph{maximal} CM (or maximal SCM) if it is not
strictly contained in a larger CM (SCM) relation.
%\end{itemize}
\end{defn}
\begin{prop}\label{proposition.id_scm}
The identity relation on $\mathbb{R}^n$, defied by $R_{id} = \{(u,u): u\in\mathbb{R}^n\}$, is strictly cyclically monotonic.
\end{prop}

\begin{proof}

Consider $u_1,\ldots,u_N\in\mathbb{R}^n$ and let $\underline{u}=[u_1^T,\ldots,u_N^T]^T$ be the stacked vector. We consider the shift operator on $Nn$-vectors,
\begin{equation}
\sigma([u_1^T,\ldots,u_N^T]^T) = [u_2^T,u_3^T,\ldots, u_N^T,u_1^T]^T.
\end{equation}
The cyclic sum \eqref{eq:CM} of the relation $R_{id}$ can be written in matrix form as $\underline{u}^T(I-\sigma)\underline{u}$. Thus, we want to make sure that the matrix $I-\sigma$ defines a positive semi-definite quadratic form, and that any eigenvector with a zero eigenvalue is of the form $[u^T,u^T,\ldots,u^T]^T$. 

The first part is equivalent to saying that $K=(I-\sigma)+(I-\sigma)^T$ is a positive semi-definite matrix. The fact that $\sigma^T=\sigma^{-1}=\sigma^{N-1}$, and the fact that all of $\sigma$'s eigenvalues are roots of unity of order $N$ imply that $K$'s eigenvalues are of the form $2-e^{j\phi}-e^{-j\phi}=2-2\cos (\phi) \ge 0$, meaning that $K$ is indeed positive semi-definite. Furthermore, if $\underline{u}$ is an eigenvector with $(I-\sigma)\underline{u}=0$, then $\sigma\underline{u}=\underline{u}$. Writing the last equation by coordinates implies that $\underline{u}$ is indeed of the form $[u^T,u^T,\ldots,u^T]^T$. This shows that $R_{id}$ is SCM. 
\end{proof}

We now show that under certain assumptions, the relations induced by steady-state input output pairs of a stable linear system also defines a CM relation.  Consider the linear system
\begin{equation} \label{eq:LinearSystem}
\begin{cases}
\dot{x} = Ax+Bu+P\mathrm{w}\\
y=Cx+Du+G\mathrm{w},
\end{cases}
\end{equation}
where $u$ is the control input, $x$ is the state, $y$ is the output, and $\mathrm{w}$ is some constant exogenous input, and assume $A$ is Hurwitz. %. We will deal with the case in which $A$ is Hurwitz. %, as if $A$ has eigenvalues in the right half-plane the system is even marginally stable, and the case in which $A$ has eigenvalues on the imaginary axis is trickier. 
For the system \eqref{eq:LinearSystem}, a constant input $u_{ss}$ yields the steady-state output \cite{Linear_System_Transfer_Function} 
%tells us that if one inserts a constant input $u_{ss}$, the steady-state output of (\ref{eq:LinearSystem}) is
\begin{equation}
y_{ss} = (-CA^{-1}B+D)u_{ss} + (-CA^{-1}P+G)\mathrm{w}.
\end{equation}
In other words, the steady-state input-output relation is given by $R=\{(u,Su+v):u\in\mathbb{R}^n\}$, where $S=CA^{-1}B+D$, and $v = (CA^{-1}P+G)\mathrm{w}$ is a constant vector. %We now show that under certain assumptions for the the system \eqref{eq:LinearSystem}, the linear relation $R$ is cyclically monotone. 
\begin{thm}{(Linear Relations)}\label{thm:LinearRelations}
The relation $R$ defined above is CM if the matrix $S$ is symmetric positive semi-definite. Furthermore, the relation $R$ is SCM if $S$ is positive-definite.
\end{thm} 

\begin{proof}
We first define the relation $R_0 = \{(u,Su): u\in\mathbb{R}^n\}$. We claim that $R$ and $R_0$ have the same cyclic sums (appearing in (\ref{eq:CM})). Indeed, if one takes $\{(u_i,y_i)\}_{i=1}^N\in R$, its cyclic sum is
\begin{align*}
	\sum_{i=1}^N y_i^T(u_i-u_{i-1}) &= \sum_{i=1}^N (Su_i+v)^T(u_i-u_{i-1}) \\
	&= \sum_{i=1}^N (Su_i)^T(u_i-u_{i-1}),
\end{align*}  
since $\sum_{i=1}^N v^T(u_i-u_{i-1}) = v^T\sum_{i=1}^N(u_i - u_{i-1}) = 0$. Thus, $R$ is CM (or SCM) if and only if $R_0$ is.

Consider now the relation $R_0$ and the identity relation $R_{id}=\{(u,u):u\in\mathbb{R}^n\}$. We claim that if $S$ is positive semi-definite, then cyclic sums of $R_0$ are also cyclic sums of $R_{id}$. %, so we'll be able to complete the proof by Proposition \ref{proposition.id_scm}.  
For positive semi-definite $S$, there exists a matrix $L$ such that $S = L^TL$.  When $S$ is positive-definite, $L$ is also invertible. % for some matrix $L$. Furthermore, if $S$ is positive definite, then $L$ is invertible. 
It is now straightforward to show that the cyclic sum of $\{(u_i,Su_i)\}_{i=1}^N\in R_0$ is equal to the cyclic sum of $\{(Lu_i,Lu_i)\}_{i=1}^N\in R_{id}$, since
\begin{equation*}
u_i^T S(u_i-u_{i-1}) = (Lu_i)^T (Lu_i-Lu_{i-1}).
\end{equation*}
By Proposition \ref{proposition.id_scm}, the relation $R_0$ is CM whenever $S$ is positive semi-definite, and SCM when $S$ is positive-definite.
\end{proof}

\begin{exam}
One can consider a linear system of the form \eqref{eq:LinearSystem} where $C=B^T, A=A^T, D=D^T, G=P=0$. Such systems are known as \emph{symmetric systems}, and their passivity was extensively researched \cite{Willems,MeisamiazadMohammadpourGrigoriadis}. Note that by Theorem \ref{thm:LinearRelations}, if $A$ is negative definite and $D-B^TAB$ is positive-definite, then the input-output relation of the system is cyclically monotonic.
\end{exam}

\subsection{Rockefeller's Theorem and Implications}
In the paper \cite{SISO_Paper}, monotone relations are so vital because they give rise to convex functions, namely by  integration. To show that CM relations are indeed generalizations of the monotone relations, we first state and prove the main theorem about CM relations, due to Rockafellar.
\begin{thm} \label{thm:Rockafellars} [\cite{Rockafeller}]
A relation $R\subset\mathbb{R}^n\times\mathbb{R}^n$ is CM if and only if it is contained in the sub-gradient of a convex function $\psi:\mathbb{R}^n\rightarrow\mathbb{R}$. Furthermore, $R$ is SCM if and only if it is contained in the sub-gradient of a strictly convex function.
\end{thm}

\begin{proof}
%The reader is referred to \cite{Rockafeller} for a complete proof. 
We provide here a sketch of the proof for the convenience of the reader (the complete proof can be found in \cite{Rockafeller}).
On one hand, if $R\subseteq\partial\psi$, and we take pairs $\{(u_i,y_i)\}_{i=1}^N$, then by convexity of $\psi$, for each $i$,
\begin{equation} \label{eq:subgrad}
\psi(u_i) \ge \psi(u_{i-1}) + y_{i-1}^T\cdot(u_i-u_{i-1}).
\end{equation}
Summing inequality \eqref{eq:subgrad} over $i$ yields a common factor $\sum_{i}\psi(u_i)$, which cancels, leaving the inequality
\[
0 \ge \sum_{i=1}^N y_{i-1}^T\cdot(u_i-u_{i-1}),
\]
the desired result. Furthermore, if the function $\psi$ was strictly convex, and not all $u_i$ are the same, then at least one of the inequalities (\ref{eq:subgrad}) is strict, making all later inequalities strict, giving us an SCM relation.

On the other hand, if $R$ is a CM relation, one can define the function
\small
\begin{equation}
\psi(u)=\sup\Bigg\{ y_0^T(u_1-u_0)+y_1^T(u_2-u_1)+\cdots+y_m^T(u-u_m)\Bigg\},
\end{equation}
\normalsize
where the supremum is taken over all pairs $\{(u_i,y_i)\}_{i=1}^m\in R$, and $m$ is allowed to change arbitrarily. When the pairs are fixed, we obtain a linear function in $u$, so $\psi$ is convex as a supremum of convex functions \cite{Rockafellar1}. One can also show that the sub-gradient of $\psi$ contains $R$. It is also straightforward to show the if $R$ is SCM, then $\psi$ is strictly convex. %The reader is referred to \cite{Rockafeller} for the rest of the proof.
\end{proof}

We now prove that CM relations are indeed a generalization of monotone relations. We cite two lemmas regarding monotone functions on $\mathbb{R}$. Their proof can be found in the appendix.
\begin{lem}\label{lemma.monotoneFunctions}
Let $f:I\to\mathbb{R}$ be a monotone increasing function defined on an interval $I\subseteq\mathbb{R}$. Fix some $x_0\in I$ and define $g:I\to\mathbb{R}$ by $g(x)=\int_{x_0}^{x}{f(t)dt}$. Then $g$ is a convex function, and the sub-differential $\partial g$ contains the set $\{(x,f(x)): x\in I\}$. Furthermore, if $f$ is strictly ascending, then $g$ is strictly convex.
\end{lem}

%And a second lemma:
\begin{lem}\label{lemma.monotoneFunctions2}
Let $R$ be a monotone relation. Then for all but countably many $u\in \mathbb{R}$, there is at most one $y\in \mathbb{R}$ such that $(u,y)\in R$.
\end{lem}

We are now ready to prove the main result. 
\begin{thm}\label{thm.CM_Generalizes_Monotone}
Let $R\subseteq\mathbb{R}\times\mathbb{R}$. Then $R$ is a (strictly) CM relation if and only if it is a (strictly) monotone relation.
%Furthermore, it is an SCM relation if and only if it is a strictly monotone relation
\end{thm}
\begin{proof}
We first show that CM (or SCM) relations are monotone (or strictly monotone). Indeed, take two arbitrary pairs $(u_1,y_1),(u_2,y_2)\in R$ and consider their cyclic sum, %. % associated with this pair. 
%Recalling that we denoted $u_0 = u_2$ and that transpose has no effect, the cyclic sum becomes:
\[
y_1^T(u_1-u_2) + y_2^T(u_2-u_1) = (y_2-y_1)(u_2-u_1) \geq 0.
\]
This cyclic sum is non-negative (or positive if the $u$'s are different and the relation is SCM), showing that $R$ is monotonic (or strictly monotonic).

Conversely, assume that $R$ is monotonic. We can extend $R$ to a maximal monotonic relation so that the $u$'s can be taken from an interval $I$. For each $u \in I$ we choose some $f(u)$ such that $(u,f(u))\in R$. We consider $f:I\rightarrow \mathbb{R}$ as a function, which is monotonic ascending. By Lemma \ref{lemma.monotoneFunctions} we have some convex function $g$ whose sub-differential contains all pairs $(u,f(u))$. 

However, if we change our choices $f$, then the function $g$ remains unaltered, as Lemma \ref{lemma.monotoneFunctions2} implies that $f$ will only change on a countable set. But we can get all of $R$ by changing the choices $f$, meaning that the sub-differential of $g$ contains $R$. Thus $R$ is cyclically monotonic by Rockefellar's Theorem \ref{thm:Rockafellars}. Furthermore, if the original relation was strictly monotonic, then the function $g$ is strictly convex, and the relation is strictly cyclically monotonic.
\end{proof}

\subsection{Cyclically Monotone Relations in Systems Theory}
Recall that in \cite{SISO_Paper}, the notion of monotone relations was used to establish network convergence. They did so by demanding that the collection of all steady-state pairs of input and output would constitute a monotonic relation. This assumption can be generalized to the case of MIMO agents by requiring that the relations will be cyclically monotonic instead. 
In that paper, the demand for a monotonic steady-state input-output relation was juxtaposed with a passivity requirement (at any steady-state) to assure convergence. Thus, we oblige and do the same here.

\begin{defn}
A system $\Sigma$ is \emph{maximal equilibrium-independent cyclically monotone}  (output strictly) passive (MEICMP) if
\begin{itemize}
\item[i)] for every steady-state input-output pair $(\mathrm{u},\mathrm{y})$, the system $\Sigma$ is (output strictly) passive  with respect to  $\mathrm{u}$ and $\mathrm{y}$;
\item[ii)] the set of all steady-state input-output pairs, $R$, is (strictly) cyclically monotonic;
\item[iii)] there is no strictly monotonic relation $S$ which strictly contains $R$.
\end{itemize}
If $R$ is strictly cyclically monotone, we say that $\Sigma$ is \emph{maximal equilibrium-independent strictly cyclically monotone} (output strictly) passive (MEISCMP).
\end{defn}
%\begin{rem}
%We abbreviate the terms above for ease of reading. 
%\begin{itemize}
%\item Maximal equilibrium-independent cyclically monotone passive relations will be called MEICMP relations.
%\item Maximal equilibrium-independent strictly cyclically monotone passive relations will be called MEISCMP relations.
%\item Output strictly passive relation gain the "output-strictly" prefix.
%\end{itemize}
%\end{rem}

Now, suppose that $\Sigma$ is MEICMP. We consider the set-valued map $k(\mathrm{u})$ which evaluates a vector $\mathrm{u}$ to the set of all vectors $\mathrm{y}$ such that ($\mathrm{u,y}$) is a steady-state input-output pair. Then, by assumption, $\{(\mathrm{u,y}):y\in k(\mathrm{u})\}$, is a maximal CM relation. Rockefellar's Theorem implies that there is some convex function $K$ such that $k$ is the sub-differential of $K$.

\begin{exam} \label{exam:NonlinearIntegratorExample}
Consider a nonlinear MIMO integrator $\dot{x}=u, y=\psi(x)$ where $\psi$ is a passive static nonlinearity. The steady-state input-output relation is $\{(0,y): y\in\mathbb{R}^d\}$, which is maximal strictly CM, and passivity follows from the passivity of the nonlinearity $\psi$ \cite{Khalil}. \textcolor{black}{The inregral function in this case is $I_0$, i.e. $I_0(x)=\begin{cases} 0 & x=0 \\ \infty & \text{else}} Thus the system is MEISCMP. In this work, we will be particularly interested in nonlinearities $\psi$ that are the gradient of a strictly convex function $P$.
\end{exam}

\subsection{Steady-State Output Agreement of $(\Sigma,\Pi,\mathcal{G})$ using CM Relations}\label{subsec.ProvingConvergence}
Up to now, we've seen that CM relations are good counterparts to monotone relations. From now on, our goal is to extend this idea to show that MEICMP systems are good counterparts of MEIP systems, proving that they allow analysis and synthesis results in the spirit of \cite{SISO_Paper,LCSS_Paper} for MIMO systems. In the following subsections, we will have a network system as in Figure \ref{ClosedLoop}, and we will assume that the agents are output-strictly MEICMP. %We'll use the convex functions associated with the input-output steady-state relations to achieve a kind of \textit{inverse optimality}.

From now on, we let $k_{i}(\mathrm{u}_i)$ be the input-output steady-state relation of the $i$-th node dynamic system, and let $k(\mathrm{u})$ be the stacked input-output steady-state relation.
The domain of $k_{i}$ will be denoted by $\mathcal{U}_i$, and it's range will be subsets of $\mathcal{Y}_i$. Similarly, the domain of $k$ will be denoted by $\mathcal{U}=\prod_{i=1}^{|\mathbb{V}|} \mathcal{U}_i$, and its range by $\mathcal{Y}=\prod_{i=1}^{|\mathbb{V}|} \mathcal{Y}_i$.

The particular goal of this subsection is to study the \emph{output agreement} problem of the system $(\Sigma, \Pi, \G)$. That is, we would like the outputs $y_i(t)$ of all nodes to converge to the same value. %deal with the problem of output agreement, meaning that the outputs $y_i(t)$ of all nodes should converge to the same value. 
Equivalently, all controller inputs $\zeta_e(t)$ should converge to the same value. Thus, a reasonable choice for the controller is a system which only has $0$ as a steady state input, i.e. the nonlinear integrator system from Example \ref{exam:NonlinearIntegratorExample}. This choice is also motivated by \cite{Wieland_Sepulchre_Allgower} which derived the internal model principle for output synchronization of networked systems.

Therefore, for this subsection alone, we make the following assumption. 
\begin{assump} \label{AssumptionOutputSynchro}
The controllers of the network system are MEISCMP and of the form
\begin{align}
\Pi_e:\ \begin{cases}
\dot{\eta}_e = \zeta_e \\
\mu_e = \psi_e (\eta_e) \end{cases}
\end{align}
where $\psi_e$ is the gradient of a \emph{strictly} convex function $P_e$. We further assume that the stacked vector $\eta(0) \in \IM(\E^T)$. % lies in the range of $\E^T$, meaning that $\eta(t)$ lies there for all times $t$.  
The stacked controller is $\dot{\eta}=\zeta,\ \mu=\psi(\eta)$ where $\psi=\nabla P$ and $P=\sum_{i=1}^{|\EE|} P_i$ is a strictly convex function.
\end{assump}

With this assumption in place, we study properties of network systems comprised of these MEISCMP agents and controllers.

\subsubsection{Node Dynamics}
First we would like to understand the properties of the steady-state solutions for network systems comprised of MEICMP agents.  Indeed, in steady-state we require that $\dot{x}_i=0$ and $\dot{\eta}_e=0$ for all agents and controllers.  This leads to the following result.
\begin{prop}\label{prop.netfeasible}
If the network system $(\Sigma, \Pi, \mathcal{G})$ admits a steady-state solution, then the steady-state solution $\mathrm{u,y}$ satisfies $\mathrm{y} \in k(\mathrm{u})$, $\mathrm{u} \in \IM(\mathcal{E})$, and $\mathrm{y} \in \ker(\mathcal{E}^T)$.
\end{prop}
\begin{proof}
The steady-state criterion for both the dynamics and the controller implies that $\mathrm{y}\in k(\mathrm{u})$ and that $\mathrm{y}\in\ker(\mathcal{E}^T)$. The definition $u(t)=-\mathcal{E}\mu(t)$ implies that $u(t)\in \IM(\mathcal{E})$, and thus $\mathrm{u}\in \IM(\mathcal{E})$. 
\end{proof}
We are now able to relate these steady-state solutions to the solutions of a network optimization problem.  Recall that the integral of the steady-state input-output relation for system $i$, $k_i(\mathrm{u}_i)$, is convex.  In this direction, let $K_i:\mathbb{R}^d\rightarrow\mathbb{R}$ be such that the sub-differential of $K_i$ satisfies $\partial K_i=k_i$.  From Proposition \ref{prop.netfeasible}, we must have that $\mathrm{u} \in \IM(\mathcal{E})$, and thus we can interpret the steady-state control input as a network divergence, and the function $K_i$ as a divergence cost.  

Similarly, for the controllers defined in Assumption \ref{AssumptionOutputSynchro}, the input-output relation for each edge can be defined as $\gamma=\{(0,\mathrm{y})|\mathrm{y}\in(\mathbb{R}^d)^{|\V|}\}$, i.e., the MIMO integrator.  The integral function of $\gamma_e$ is $\Gamma_e = I_0$, and the controller variables $\zeta$ may be interpreted as tensions in the network.

Given the functions $K_i$ and $\Gamma_e$ defined above, we can also define their convex conjugates.  For example, $K^{\star}_i(\mathrm{y_i}) = \sup_{\mathrm{u}_i\in\mathcal{U}_i} \{y_i^Tu_i - K_i(\mathrm{u}_i)\}$. 
We denote $K(\mathrm{u}) = \sum_{i=1}^{|\mathbb{V}|} K_i(\mathrm{u}_i)$ and $K^{\star}(\mathrm{y}) = \sum_{i=1}^{|\mathbb{V}|} K_i^{\star}(\mathrm{y}_i)$. We are now prepared to present one of the main results.
\begin{thm}[Inverse Optimality of Output Agreement] \label {InverseOptimality}
Consider a network system $(\Sigma, \Pi, \G)$, in which all nodal systems are MEICMP and the controllers satisfy Assumption \ref{AssumptionOutputSynchro}. Let $\mathrm{u,y}$ be a steady-state input-output pair. Then
\begin{enumerate} 
\item $\mathrm{u}$ is an optimal solution of ({OFP}).
\item $\mathrm{y}$ is an optimal solution of ({OPP}).
\item $K(\mathrm{u}) + K^{\star}(\mathrm{y}) = 0$. 
\end{enumerate}
\end{thm}
\begin{rem}
In this case, we have $\Gamma(\zeta)=I_0$, forcing $\zeta$ to be equal to 0 when solving {(OPP)}. Similarly, we have $\Gamma^\star(\mu)=0$, so we do not care about the value of $\mu$ when solving {(OFP)}. Thus, we may abuse the terminology and talk about $\mathrm u$ and $\mathrm y$ as solutions of {(OFP)} and {(OPP)}.
\end{rem}
We now prove the theorem:
\begin{proof}
As $\mathrm{u,y}$ are steady-state input-output pairs, we have that $\mathcal{E}^T\mathrm{y} = 0$, and for the appropriate steady-state $\mathrm{\mu}$ we have $\mathrm{u}+\mathcal{E}\mathrm{\mu} = 0$, i.e. both $\mathrm{u}$ and $\mathrm{y}$ are feasible for (OFP) and (OPP), correspondingly.\\
Now, by definition of the input-output steady-state relation, $\mathrm{y}\in k(\mathrm{u})$. Because $\partial K = k$, we know that the supremum defining the dual function $K^{\star}(\mathrm{y})$ is achieved at $\mathrm{u}$ \cite{Rockafellar1}. Thus, $\mathrm{u^T y} = 0$ implies
$$K^\star(\mathrm{y}) = \mathrm{u^T y} - K(\mathrm{u}) = -K(\mathrm{u}),$$
and by Corollary \ref{cor.feasopt} the result is proved.  
\end{proof}

This connection between steady-state input-output pairs and the convex optimization problems gives us a path to understanding the steady states the system has traversing through techniques of convex analysis. In particular, we can prove existence and uniqueness of solutions by examining certain properties of the map $k$.

\begin{cor}
Suppose all node dynamics in the system are maximally equilibrium-independent cyclically-monotone passive, and that $\mathcal{U}_i = \mathbb{R}^d = \mathcal{Y}_i$. Then an output agreement steady-state exists.
\end{cor}
\begin{proof}
The range and domain assumption imply that both (OFP) and (OPP) are feasible. Let $(\mathrm{u,y})$ be optimal dual solutions to (OFP) and (OPP). Duality of the solutions gives that $\mathrm{y}\in k(\mathrm{u})=\partial K(\mathrm{u})$. Also, because $\mathrm{y}$ is feasible for (OPP), we must have $\E^T\mathrm{y}=0$, meaning that  $\mathrm{y}$ describes an output agreement. Thus we have a steady-state input-output pair describing output agreement.
\end{proof}

\begin{cor} \label{Uniqueness}
Assume that the nodal systems are maximally equilibrium-independent strictly cyclically monotone passive. Suppose furthermore that the input-output steady state-relations $k_i$ are functions. 
Even further, suppose that whenever $u^{1},u^{2},...$ is a sequence of points in $\mathcal{U}_i$ converging to a boundary point of $\mathcal{U}_i$, we have $\lim_{\ell\rightarrow\infty} |k_i(u^{\ell})|=\infty$. Then there exists no more than one steady-state input-output pair $(\mathrm{u},\mathrm{y})$.
\end{cor}

\begin{proof}
The assumptions yield that the function $K$ are differentiable and essentially smooth convex functions \cite[p.~251]{Rockafellar1}. Thus, strict convexity implies that (OFP) has at most one solution, and if such a solution exists, the dual problem has a solution.
\end{proof}

\begin{cor}
Assume the same assumptions as Corollary \ref{Uniqueness}. If an output agreement steady-state exists, the agreement value $\beta$ satisfies
\begin{equation}
\sum_{i=1}^{|\mathbb{V}|} k_i^{-1}(\beta) = 0.
\end{equation}
\end{cor}

\begin{proof}
Theorem 26.1 from \cite{Rockafeller} implies that $\nabla K^{\star}(\mathrm{y_i})$ exists. Replacing $\mathrm{y}$ with the vector $[\beta^T,\beta^T,\ldots,\beta^T]^T$ in $K^\star(y)$, differentiating and equating to 0 gives the desired equation.
\end{proof}

\subsubsection{ The Controllers}
As before, our goal is to understand the properties of steady-state solutions of the closed-loop system. Proposition \ref{prop.netfeasible} characterizes the steady-state inputs for the node systems. The requirement that $\dot{\eta_e}=0$ in steady-state for all controllers leads to the following result.

\begin{prop}\label{prop.contfeasible}
Suppose that Assumption \ref{AssumptionOutputSynchro} holds true. If the network system $(\Sigma, \Pi, \mathcal{G})$ admits a steady-state solution $\mathrm{u,y},\eta$, then the steady-state solution $\mathrm{u},\eta$ satisfies $\eta\in\IM(\E^T)$ and $\mathrm{u}=-\E\psi(\eta)$.
\end{prop}
\begin{proof}
First note that in steady-state, we have $\E^T\mathrm{y}=0$, meaning that the steady-state controller input $\zeta$ nulls and the controller state $\eta$ is constant.
Assumption \ref{AssumptionOutputSynchro} states that the signal $\eta(t)$ lies in $\IM(\E^T)$, implying that $\eta\in\IM(\E^T)$. Furthermore, the definitions $\mu(t)=\psi(\eta(t))$ and $u(t)=-\E\mu(t)$ imply that $\mathrm{u}=-\E\psi(\eta)$.
\end{proof}

Note that given a steady-state input $\mathrm u$ to the closed-loop system, Proposition \ref{prop.contfeasible} gives a criterion for $\eta$ to be a steady-state producing it. Because $\mathrm u$ is usually unknown when designing the controller, our goal is to find a controller design allowing the conditions in Proposition \ref{prop.contfeasible} to be fulfilled, not matter the value of $\mathrm u\in\IM(\E)$. 

We wish to connect the controller dynamics to network optimization problems. The (OPP) and (OFP) problems are slightly augmented, as using $\eta$ is more beneficial than using $\zeta$ for the output agreement case. Specifically, we recall that $P$ is an integral function of $\psi$, and we consider $\mathrm{u}$ as a fixed constant vector. %,  fix $\mathrm{u}$ and define:

\begin{center}
\begin{tabular}{ c||c }
 \textbf{Controller Form of (OPP)}  & \textbf{Controller Form of (OFP)}   \\
 (COPP) & (COFP)  \\\hline
 $ \begin{array}{cl} \underset{\mathrm{\eta,\mathrm{v}}}{\min} &P(\eta)+ \mathrm{u^Tv}\\
s.t.&{\E}^T\mathrm{v} = \eta 
\end{array} $&  $ \begin{array}{cl}\underset{\mathrm{\mu}}{\min}& P^\star(\mu) \\
s.t. &\mathrm{u} = -{\E}\mu.
\end{array} $ 
\end{tabular}
\end{center}

The problems above can be seen to be dual. Furthermore, we can fix some vector $\mathrm m$ such that $-\E\mathrm m=\mathrm{u}$, and then (COPP) can be written as a function of $\eta$ alone, as $\mathrm{u^Tv}=-\mathrm{m}^T\E^T\mathrm{v}=-\mathrm{m}^T\eta$. Thus we can talk about $\eta$ being a solution to (COPP).  We can now prove the main result regarding the controller design.
\begin{thm} (Controller Realization)
Consider a network system $(\Sigma, \Pi, \G)$, satisfying all the conditions in Theorem \ref{InverseOptimality}. Suppose furthermore that Assumption \ref{AssumptionOutputSynchro} holds. Then the network has an output agreement steady-state. 

Let $\mathrm{\eta}$ be the steady-state of the controller in this output agreement. Then:\begin{enumerate} 
\item $\mathrm{\eta}$ is an optimal solution of (COFP).
\item $\mathrm{\mu}=\psi(\eta)$ is an optimal solution of (COFP).
\item $P(\mathrm{\eta}) + P^{\star}(\mathrm{\mu}) = \mathrm{\mu}^T\mathrm{\eta}$. 
\end{enumerate}\end{thm}

The proof of the theorem is similar to the proof of Theorem \ref{InverseOptimality}. Like in that case, we use Lemma \ref{lem:ZeroEquationLemma} to draw a conclusion for our case.
\begin{cor} \label{cor:ZeroEquationLemmaController}
Let $\mu_0,\eta_0$ be two feasible solutions to (COPP) and (COFP). If $P(\eta_0)+P^\star(\mu_0)=\eta^T\mu$ then both solutions are optimal.
\end{cor}

\begin{proof}
We consider the following pair of problems, which are equivalent to (COPP) and (COFP), by defining $\Delta\mu=\mu-\mu_0$,
\begin{center}
\begin{tabular}{ c||c }
 \textbf{(COPP) Equivalent}  & \textbf{(COFP) Equivalent}   \\
\hline
 $ \begin{array}{cl} \underset{\mathrm{\eta,\mathrm{v}}}{\min} &P(\eta)- \mu_0^T\eta\\
s.t.&{\eta \in \IM(\E^T)} 
\end{array} $&  $ \begin{array}{cl}\underset{\mathrm{\Delta\mu}}{\min}& P^\star(\mu_0+\Delta\mu) \\
s.t. &\Delta\mu \in \ker(\E).
\end{array} $ 
\end{tabular}
\end{center}
Now these are just a special case of (OPP) and (OFP), under the change of notation from $\zeta$ to $\eta$ and the choices $K=I_0$ and $\Gamma(\zeta)=P(\zeta)-\mu_0^T\zeta$. Thus we can apply Lemma \ref{lem:ZeroEquationLemma} and conclude the proof of the theorem, where we note that $P^\star(\mu_0)=\Gamma^\star(0)$, and that $\mu_0$ is optimal for (COFP) if and only if $0$ is optimal for the equivalent problem.
\end{proof}

We are now ready to prove the theorem.
\begin{proof}
Let $\mathrm{u}$ be the corresponding steady-state input to the agents. By Proposition \ref{prop.contfeasible}, we know that $\eta$ satisfies $\mathrm{u}=-\E\psi(\eta)$, and that $\eta\in\IM(\E^T)$, meaning that $\eta$ is a feasible solution to (COPP), and that $\mu$ is a feasible solution to (COFP). Moreover, We know that $\mu=\psi(\eta)$ and that $\psi=\nabla P$, so the supremum in the definition of $P^\star(\mu)$ is achieved at $\eta$ \cite{Rockafellar1}. Thus we have $P^\star(\mu)=\mu^T\eta-P(\eta)$, or $P(\eta)+P^\star(\mu)=\mu^T\eta$. Thus, Corollary \ref{cor:ZeroEquationLemmaController} implies that both $\eta$ and $\mu$ are optimal solutions to (COPP) and (COFP).
\end{proof}

\subsubsection{Closing The Loop}
We can now prove the main theorem of this paper.
\begin{thm} \label{thm:ConvergenceOfClosedLoopSystemSimp}
Consider the network system $(\Sigma, \Pi, \G)$, and suppose all node dynamics are maximally equilibrium-independent cyclically monotone passive and Assumption \ref{AssumptionOutputSynchro} holds. Then there exists constant vectors $\mathrm{u,y,\mu,\eta}$ such that $\lim_{t\rightarrow\infty} u(t) = \mathrm{u}$, $\lim_{t\rightarrow\infty} y(t) = \mathrm{y}$, $\lim_{t\rightarrow\infty} \mu(t) = \mathrm{\mu}$, and $\lim_{t\rightarrow\infty} \eta(t) = \mathrm{\eta}$. Moreover, $(\mathrm{u,y})$ are optimal solutions to (OPP) and (OFP), and $(\eta,\mu)$ are optimal solutions to (COPP),(COFP).
In particular, the network converges to output agreement, meaning that there exists some vector $\beta$ such that $\lim_{t\rightarrow\infty} y(t) = \beta\otimes\mathbbm{1}$.
\end{thm}
\begin{proof}
Our assumption implies that all four optimization problems (OPP), (OFP), (COPP), and (COFP) have solutions, meaning that a steady-state solution exists. The equilibrium-independent passivity assumption implies that there are storage functions $S_i$ (for $i\in\mathbb{V}$) and $W_e$ (for $e\in\mathbb{E}$), such that
\begin{align}
 \dot{S_i} &\le -\rho_i ||y_i(t)-\mathrm{y}_i||^2 + (y_i(t) - \mathrm{y}_i)^T(u_i(t) - \mathrm{u}_i)\\
\dot{W_e} &\le  (\mu_e(t) - \mathrm{\mu}_e)^T(\zeta_e(t) - \mathrm{\zeta}_e) \nonumber ,
\end{align}
where $\mathrm{\zeta}_e = 0$. Theorem \ref{ConvergencePassivity} implies that $y(t)$ converges to $\mathrm{y}$, implying that $\zeta(t)$ converges to $0 = \mathrm{\zeta} = \mathcal{E}^T\mathrm{y}$. Integrating implies that $\eta(t)$ converges to some $\mathrm{\eta}$.

In turn, this implies that $\mu(t)$ converges to $\mathrm{\mu} = \psi(\mathrm{\eta})$ and that $u(t)$ converges to $\mathrm{u} = -\mathcal{E}\mathrm{\mu}$. It is clear that $(\mathrm{u,y})$ is a steady-state pair, and furthermore that $(\mathrm{u,y,\eta,\mu,\zeta})=0$ are satisfy the conditions in Propositions \ref{prop.netfeasible} and \ref{prop.contfeasible}. This concludes the proof of the theorem.
\end{proof}

This chapter can be summarized in the following way - the signals in the closed-loop network depicted in Figure \ref{ClosedLoop} have static counterparts in the field of network optimization theory. The signals in the node dynamics have associates in (OPP) and (OFP), while the controller dynamics can be reflected by (COPP) and (COFP). 

\subsection{Analysis For General Network Systems}
Using the notion of maximal equilibrium-independent cyclically monotonic systems, we can establish a convex optimization framework for more general systems. We now relax the integrator structure of Assumption \ref{AssumptionOutputSynchro}, and consider controllers of the form given in \eqref{eq:EdgeSystemsODE}.

%We reconsider more general forms of controllers, as appearing in \eqref{eq:EdgeSystemsODE}:
%\begin{align*} \Pi_k :\begin{cases}
%\dot{\eta}_k = \phi_k(\eta_k,\zeta_k), \\
%\mu_k = \psi_k(\eta_k,\zeta_k)   , 
%\end{cases}
%k\in\mathbb{E}.
%\end{align*}

\begin{assump} \label{assump:GeneralController}
The controllers (\ref{eq:EdgeSystemsODE}) are maximally equilibrium-independent cyclically monotone passive.
\end{assump}

From now on, we assume that Assumption \ref{assump:GeneralController} holds. We denote the input-output steady-state relation for the edge $e$ by $\gamma_e$, and let $\gamma(\zeta)$ be the stacked input-output steady-state relation for the controllers. The domain of $\gamma_e$ will be denoted by $\mathcal{Z}_e$, and it's range will be subsets of $\mathcal{M}_e$. Similarly, the domain of $\gamma$ will be denoted by $\mathcal{Z}=\prod_{e=1}^{|\mathbb{E}|} \mathcal{Z}_e$, and its range by $\mathcal{M}=\prod_{e=1}^{|\mathbb{E}|} \mathcal{M}_e$. We denote the corresponding integral functions by $\Gamma_e$ and $\Gamma$. The formalism of the previous sections can be implemented immediately, where one now considers the general forms of (OPP) and (OFP)

\begin{thm} \label{thm:GeneralConvergence}
Consider the network system $(\Sigma, \Pi, \G)$, and suppose all node dynamics are maximally equilibrium-independent cyclically monotone passive and Assumption \ref{assump:GeneralController} holds. Then there exists constant vectors $\mathrm{u,y,\mu,\eta}$ such that $\lim_{t\rightarrow\infty} u(t) = \mathrm{u}$, $\lim_{t\rightarrow\infty} y(t) = \mathrm{y}$, $\lim_{t\rightarrow\infty} \mu(t) = \mathrm{\mu}$, and $\lim_{t\rightarrow\infty} \eta(t) = \mathrm{\eta}$. Moreover, $(\mathrm u, \zeta)$ and $(\mathrm y,\zeta)$ form optimal dual solutions to (OPP) and (OFP).
\end{thm}

As for Theorem \ref{thm:ConvergenceOfClosedLoopSystemSimp}, we split the proof into two parts. The first part shows that solutions of the optimization problems are steady-states of the closed-loop system, while the second part uses passivity to show that the closed-loop system must converge to a steady-state.

\begin{lem} [Inverse Optimality]
Suppose that the assumptions of Theorem \ref{thm:GeneralConvergence} hold. Let $\mathrm{u,y,\zeta,\mu}$ be constant signals such that $\E^T \mathrm{y} = \mathrm{\zeta}$ and $\mathrm{u}=-\E\mu$. Then the following conditions are equivalent:
\begin{enumerate}
\item[i)] $\mathrm{u,y,\zeta,\mu}$ are steady-states of the closed-loop system;
\item[ii)] $(\mathrm{y,\zeta})$ is an optimal solution to (OPP), $(\mathrm{u,\mu})$ is the dual optimal solution to (OFP), and $K(\mathrm{u})+K^\star(\mathrm{y})+\Gamma(\zeta)+\Gamma^\star(\mu)=0$.
\end{enumerate}
\end{lem}

\begin{proof}
We first note, as before, that $\mathrm{u,y,\zeta,\mu}$ is a steady-state of the closed-loop system if and only if 
\[
\E^T\mathrm{y}=\zeta,\ \mu\in\gamma(\zeta),\ \mathrm{u} = -\E\mu,\ \mathrm{y} \in k(\mathrm{u}).
\]
Also, the dual of $F=K+\Gamma^\star$ is $K^\star+\Gamma$ and $\partial(K^\star+\Gamma)=(k^{-1},\gamma)$ \cite{Rockafellar1}.

Suppose that (1) holds. We know that $(\mathrm u,\mu)^T(\mathrm y,\zeta) = 0$ and that the supremum defined in $F^\star(\mathrm y,\zeta)$ is achieved in $(\mathrm u,\mu)$. Thus
\[
F^\star(\mathrm y, \zeta) = (\mathrm y,\zeta)^T(\mathrm u,\mu) - F(\mathrm u,\mu) = -F(\mathrm u,\mu),
\]
implying that $K(\mathrm u)+K^\star(\mathrm y)+\Gamma(\zeta)+\Gamma^\star(\mu)=0$, so the proof is complete by Lemma \ref{lem:ZeroEquationLemma}.

Now, suppose that (2) holds. By duality of optimal solutions, we obtain 
\[
(\mathrm{u,\mu})\in \partial(K^\star+\Gamma)|_{(\mathrm{y,\zeta})}=(k^{-1}(\mathrm{y}),\gamma(\zeta)),
\]
meaning that $\mathrm{y}=k(\mathrm{u})$ and $\mu=\gamma(\zeta)$. Together with the assumption that $\mathrm u=-\E\mu$ and $\zeta=\E^T\mathrm y$, we conclude that the solution is indeed a steady-state of the problem.
\end{proof}
We can now prove Theorem \ref{thm:GeneralConvergence}.
\begin{proof}
The lemma implies that the optimal solutions to the optimization problems correspond the equilibria in the system. The rest of the proof, dealing with convergence of the signals $u(t),y(t),\mu(t),\zeta(t)$ to $\mathrm{u,y,\mu,\zeta}$, is exactly the same as the proof of Theorem \ref{thm:ConvergenceOfClosedLoopSystemSimp}, using passivity assumptions on the node and controller dynamics, as well as the basic convergence result Theorem \ref{ConvergencePassivity}.
\end{proof}
\fi
\section{Synthesis of Multi-Agent Systems Using Network Optimization Theory
\label{sec:Synthesis}}

Up to now, the established connection between network optimization
and passivity-based cooperative control only gave us analysis results,
and did not help us derive a synthesis procedure. %, needed to apply
%the theory to real-world systems. 
We now deal with the problem of forcing a certain steady-state on the closed-loop system. This is
done by an appropriate design of the edge controllers $\Pi_{e}$.

We assume the agent dynamics are given and are MEICMP. %dynamical systems for the agents, having $d$ inputs
%and outputs. 
The synthesis problem can now be formulated. % as below:
\begin{prob}
\label{prob:Synthesis}Let $\mathrm{y}^{\star}\in(\R^{d})^{|\V|}$ be some
vector. 
\begin{enumerate}
\item Find a computationally feasible criterion assuring the existence of controllers
$\{\Pi_{e}\}_{e\in\EE}$, such that the output of the system $(\Sigma,\Pi,\G)$ has ${\mathrm y}^{\star}$ as a steady-state.
\item In case ${\mathrm y}^{\star}$ satisfies the criterion, find a construction
for $\{\Pi_{e}\}_{e\in\EE}$, that makes the system converge to ${\mathrm y}^{\star}$.
\end{enumerate}
\end{prob}
This section has four parts. Subsection \ref{subsec:CharacterizingForcibleSteadyStates} deals with solving part
1 of the Problem $\ref{prob:Synthesis}$. Subsection \ref{subsec:GlobalAsympConvergenceSynthesis} deals with solving
the second part of the same problem. Subsection \ref{subsec:FormationReconfiguration} deals with different control objectives ${\mathrm y}^{\star}$,
namely by prescribing a procedure which uses a solution for some $y_{1}^{\star}$
to find a solution for $y_{2}^{\star}$ by augmenting
the controller. Finally, subsection \ref{subseq:LeadingAgentsSynthesis} addresses outputs that do not satisfy the desired synthesis criteria. % found by minor plant augmentations.

%The results of subsections 3.2.1-3.2.3 were incorporated into \cite{LCSS_Paper} and submitted to IEEE Control Systems Letters (L-CSS)., with a possibility to present in IEEE CDC 2017.

%In this section, fix the nodal dynamical systems $(\Sigma_{i})_{i\in\V}$
%to be MEICMP systems. 
As before, we denote the input-output steady-state
relations of the nodes by $k_{i}$, and their integral functions by $K_{i}$. 
%and their sum by $K$. 
We choose the controllers to be output-strictly
MEICMP, so we can discuss their input-output steady-state relations
$\gamma_{e}$ and their integral functions by $\Gamma_{e}$. %, and their
%sum by $\Gamma$.

\subsection{Characterizing Forcible Steady-States}\label{subsec:CharacterizingForcibleSteadyStates}

The result of Theorem \ref{thm.Convergence} helps us predict the steady-state outputs
of the closed loop by solving the optimal potential problem (OPP).
The outline to the solution of Problem \ref{prob:Synthesis} is given by studying
the minimizers of the optimization problem (OPP). We first prove the
following proposition.
\begin{prop}
\label{prop:Linearization}Let ${\mathrm y}^{\star}\in(\R^{d})^{|\V|}$ and let $\zeta^{\star}=\E^{T}{\mathrm y}^{\star}$. The pair $({\mathrm y}^{\star},\zeta^{\star})$
is a minimizer of (OPP) if and only if 
\begin{equation}
{\bf 0}\in k^{-1}({\mathrm y}^{\star})+\E\gamma(\zeta^{\star})\label{Linearization}.
\end{equation}
\end{prop}
Note that we demand inclusion instead of equality because the subdifferential set might include more than one element.
\begin{proof}
The network optimization problem (OPP) can be written as an unconstrained
optimization problem in terms of the variable $y$ alone; we ask to
minimize $F(y)=K^{\star}(y)+\Gamma(\E^{T}y)$. This is a convex function
of $y$, so it is minimized only where the zero vector lies in its
subdifferential \cite{Rockafeller}. Thus, by subdifferential calculus (see \cite{Rockafeller})
we obtain that ${\bf 0}\in k^{-1}({\mathrm y}^{\star})+\E\gamma(\E^{T}{\mathrm y}^{\star})$ is equivalent
to $({\mathrm y}^{\star},\zeta^{\star})$ being a minimum. Plugging in $\zeta^{\star}=\E^{T}{\mathrm y}^{\star}$
gives the desired criterion.
\end{proof}
\begin{cor}
Let ${\mathrm y}^{\star}\in(\R^{d})^{|\V|}$. Then one can choose output-strictly
MEICMP controllers $\{\Pi_{e}\}_{e\in\EE}$ so that ${\mathrm y}^{\star}$ is
a steady state of the closed-loop system if and only if $k^{-1}({\mathrm y}^{\star})\cap\IM(\E)\ne\emptyset$.
\end{cor}
\begin{proof}
If ${\mathrm y}^{\star}$ is a steady state (for some choice of controllers),
then $\eqref{Linearization}$ proves that $k^{-1}({\mathrm y}^{\star})\cap\IM(\E)\ne\emptyset$.
Conversely, if $k^{-1}({\mathrm y}^{\star})\cap\IM(\E)\ne\emptyset$,
then we can take some vector $\xi$ such that $-\E\xi\in k^{-1}({\mathrm y}^{\star})$.
If MEICMP controllers $\Pi_{e}$ are chosen so that $\gamma(\zeta^{\star})=\gamma(\E^{T}{\mathrm y}^{\star})\ni-\xi$,
then Proposition $\ref{prop:Linearization}$ implies that ${\mathrm y}^{\star}$
is a steady state of the closed-loop system. There are many ways to
choose these controllers, one of them being
\begin{equation}
\Pi_{e}:\begin{cases}
\dot{\eta_{e}}=-\eta_{e}+\zeta_{e}-(\xi_{e}+\zeta_{e}^{\star})\\
\mu_{e}=\eta_{e}
\end{cases}.\label{eq:LinearControllers}
\end{equation}
\end{proof}
\begin{rem}
The chosen controllers have a special structure - these are linear
controllers with constant exogenous inputs that make the system converge to ${\mathrm y}^{\star}$, but their
dependence on ${\mathrm y}^{\star}$ is only through the constant $\xi_{e}+\zeta_{e}^{\star}$.
This small change in the controller will make the entire system converge
to a different point. We'll emphasize this point in subsection $\ref{subsec:FormationReconfiguration}$.
\end{rem}
\textcolor{ black}{It is well-known that the set $\IM(\E)$, called the cut-space of the graph $\G$, consists of all vectors $u$ such that $\sum_{i=1}^{|\V|} u_i = 0$ \cite{Godsil_Royle}}. Thus, the first part of Problem $\ref{prob:Synthesis}$ is solved by the following result.

\begin{cor}
\label{cor:CharacterizationSynthesis}The vector $y\in(\R^{d})^{|\V|}$
is forcible as a steady-state if and only if ${\bf 0}\in\sum_{i=1}^{|\V|}k_{i}^{-1}(y_{i})$.
%This solves part 1 of the Problem $\ref{prob:Synthesis}$
\end{cor}

\subsection{Forcing Global Asymptotic Convergence}\label{subsec:GlobalAsympConvergenceSynthesis} % Using Local Strict Convexity} 

We now solve part 2 of Problem $\ref{prob:Synthesis}$, giving
criteria for controllers to provide global asymptotic convergence
and constructing controllers that satisfy these criteria.
By Theorem $\ref{thm.Convergence}$, if we take output-strictly MEICMP
controllers, then the closed-loop system converges to some ${\mathrm y}^{\star}$,
so that $({\mathrm y}^{\star},\zeta^{\star}=\E^{T}{\mathrm y}^{\star})$ are a solution
of (OPP).
\begin{cor}
If the chosen controllers are output-strictly MEICMP, and (OPP) has
only one solution $(\hat{{\mathrm y}},\hat{\zeta})$, then the closed-loop system
globally asymptotically converges to $\hat{{\mathrm y}}$.
\end{cor}
The minimization of the function $K^{\star}(y)+\Gamma(\zeta)$ appearing
in (OPP) can be divided into two parts: 
\begin{align}\label{eq:SequentialMinimization}
\min_{\E^{T}y =\zeta} [K^{\star}(y)+\Gamma(\zeta)] &= \min_{\zeta\in\IM(\E)}\min_{y:\E^{T}y=\zeta}[K^{\star}(y)+\Gamma(\zeta)] \nonumber\\ 
 &\hspace{-15pt} =\min_{\zeta\in\IM(\E)}[\Gamma(\zeta)+\min_{y:\E^{T}y=\zeta}K^{\star}(y)].
\end{align}
Our goal is to have $({\mathrm y}^{\star},\zeta^{\star}=\E^{T}{\mathrm y}^{\star})$ as the sole
minimizer of this problem. Thus we have two goals - the outer minimization
problem needs to have $\zeta^{\star}$ as a sole minimizer, and the
inner minimization problem needs to have ${\mathrm y}^{\star}$ as a sole minimizer.
The main tool we employ is the one of strict convexity. We note that
if the systems $\{\Pi_{e}\}_{e\in\EE}$ are output-strictly MEISCMP,
then the input-output relation $\gamma$ is strictly cyclically monotone,
and therefore $\Gamma$ is strictly convex. 
Let us first deal with the outer minimization problem in $\eqref{eq:SequentialMinimization}$. 
\begin{defn}
The \emph{minimal potential function} is a function $G=G_{\G,K}:\IM(\E^{T})\rightarrow\R$,
depending on the graph $\G$ and the integral functions of the agents' steady-state input-output maps, $K$, defined by
\[
G(\zeta)=\min\{K^{\star}(y)|\ \E^{T}y=\zeta\}.
\]
\end{defn}
\begin{prop}
\label{prop:OuterMinimization}Suppose that $\eqref{Linearization}$
is satisfied by the pair $({\mathrm y}^{\star},\zeta^{\star}=\E^{T}{\mathrm y}^{\star})$.
Suppose further that the function $\Gamma_{e}$ is strictly convex
in the neighborhood of $\zeta_{e}^{\star}$ for all $e\in\EE$. Then
$\zeta^{\star}$ is the unique minimizer of the outer optimization
problem in $\eqref{eq:SequentialMinimization}$, i.e., $\zeta^{\star}$
is the unique minimizer of $\Gamma(\zeta)+G(\zeta)$.
\end{prop}
\begin{proof}
Because the function $K^{\star}$ is convex, the function $G$ is
also convex (see  \cite{Appendix}). Thus $\Gamma(\zeta)+G(\zeta)$ is convex
as a sum of convex functions, and it is strictly convex near $\zeta^{\star}$.
Let $M$ be the collection of $(G+\Gamma)$'s minima. It follows from Proposition \ref{prop:Linearization} that $\zeta^{\star}\in M$.  Furthermore, the set $M$ is convex, since $G+\Gamma$ is a convex function.  Finally, there is some small neighborhood $\mathcal{U}$ of $\zeta^{\star}$
such that $M\cap\mathcal{U}$ contains no more than one point, as
$G+\Gamma$ is strictly convex in a neighborhood of $\zeta^{\star}$.
%\begin{itemize}
%\item $\zeta^{\star}\in M$, by Proposition $\ref{prop:Linearization}$.
%\item The set $M$ is convex, as $G+\Gamma$ is convex.
%\item  There is some small neighborhood $\mathcal{U}$ of $\zeta^{\star}$
%such that $M\cap\mathcal{U}$ contains no more than one point, as
%$G+\Gamma$ is strictly convex in a neighborhood of $\zeta^{\star}$.
%\end{itemize}
We claim that these facts imply that $M$ contains the single point
$\zeta^{\star}$, concluding the proof. Indeed, suppose that there's
some other $\zeta\in M$. By convexity, we have $\zeta_{t}=t\zeta+(1-t)\zeta^{\star}\in M$
for all $t\in(0,1)$, and in particular, for small $t>0$.
If $t>0$ is small enough then $\zeta_{t}\in\mathcal{U}$, as $\mathcal{U}$
is open, meaning that $\zeta_{t}\in M\cap\mathcal{U}$ for $t>0$
small. But this is impossible as $M\cap\mathcal{U}$ cannot contain more than
one point. Thus $\zeta^{\star}$ is the unique minimizer of $G+\Gamma$.
\end{proof}
Now for the inner minimization problem of $\eqref{eq:SequentialMinimization}$.
We wish that ${\mathrm y}^{\star}$ would be the unique minimizer of $K^{\star}(y)$
on the set $\{\E^{T}y=\zeta^{\star}\}=\{{\mathrm y}^{\star}+\beta\otimes\bold{1}\,|\,\beta\in\R^{d}\}$.
We consider $A:\R^{d}\rightarrow\R$ defined by $A(\beta)=K^{\star}({\mathrm y}^{\star}+\beta\otimes\bold{1})$,
and we wish that $\beta=0$ will be the unique minimizer of
$A$.

Minimizing $A$ is the same as finding $\beta$ such that ${\bf 0}\in\partial A(\beta)$.
By subdifferential calculus \cite{Rockafeller,Appendix}, we have $\partial A(\beta)=\Proj_{\ker\E^{T}}k^{-1}({\mathrm y}^{\star}+\beta\otimes\bold{1})$,
where we use $\{\beta\otimes\bold{1}:\ \beta\in\R^{d}\}=\ker\E^{T}$.
We already saw that $\Proj_{\ker\E^{T}}u=\Bigg(\frac{1}{|\V|}\sum_{1}^{|\V|}u_{i}\Bigg)\otimes\bold{1}$,
so we conclude that $0\in\partial A(\beta)$ is equivalent to $0\in\sum_{1}^{|\V|}k_{i}^{-1}(y_{i}^{\star}+\beta\otimes\bold{1})$.
Note that plugging $\beta=0$ gives the exact same condition
appearing in Corollary $\ref{cor:CharacterizationSynthesis}$, thus
if ${\mathrm y}^{\star}$ satisfies the condition in Corollary $\ref{cor:CharacterizationSynthesis}$,
then it is a solution to the inner minimization problem of $\eqref{eq:SequentialMinimization}$.
We want to make sure that it is the only minimizer. By similar methods,
we can prove the following result.
\begin{prop}
Consider the function $A(\beta)=K^{\star}({\mathrm y}^{\star}+\beta\otimes\bold{1})$.
If ${\mathrm y}^{\star}$ satisfies the condition in Corollary $\eqref{cor:CharacterizationSynthesis}$
and $A$ is strictly convex near $0$, then ${\mathrm y}^{\star}$ is the unique
minimizer of $K^{\star}(y)$ on the set $\{\E^{T}y=\zeta^{\star}\}$.
\end{prop}
The proof is exactly the same as the proof of Proposition $\eqref{prop:OuterMinimization}$. We
conclude with the main synthesis result.
\begin{thm}[Synthesis Criterion of MEICMP systems]
\label{thm:Synthesis} Consider a networked system $(\Sigma, \Pi, \G)$,, and let ${\mathrm y}^{\star}$ be the desired steady-state output.
Suppose that $\{\Pi_{e}\}_{e\in\EE}$ are output-strictly MEICMP controllers,
and denote their input-output relations by $\gamma_{e}$, and the corresponding
integral functions by $\Gamma_{e}$. Assume that the following conditions
hold:
\begin{enumerate}
\item[i)] the equation $\eqref{Linearization}$ is satisfied by the pair $({\mathrm y}^{\star},\zeta^{\star}=\E^{T}{\mathrm y}^{\star})$;
\item[ii)] for any $e\in\EE$, the function $\Gamma_{e}^{\star}$ is strictly
convex in a neighborhood of $\zeta_{e}$;
\item[iii)] the function $A:\R^{d}\rightarrow\R$, defined by $A(\beta)=\sum_{i=1}^{|\V|}K_{i}^{\star}(y_{i}^{\star}+\beta\otimes\bold{1})$, 
is strictly convex near ${\beta}=0$;
\item[iv)] the vector $0$ is in the subdifferential set $\sum_{i=1}^{|\V|}k_{i}^{-1}(y_{i}^{\star})$.
\end{enumerate}
Then the output of the closed-loop system globally asymptotically
converges to ${\mathrm y}^{\star}$. Furthermore, if the agents are output-strictly
MEICMP, we can relax our demand and require the controllers $\{\Pi_{e}\}_{e\in\EE}$
to only be MEICMP.
\end{thm}
\begin{proof}
The MEICMP assumptions imply that the closed-loop system always converges
to some solution of (OPP). The equation $\eqref{eq:SequentialMinimization}$,
together with conditions i)-iv) show that $({\mathrm y}^{\star},\zeta^{\star}=\E^{T}{\mathrm y}^{\star})$
are the unique minimizers of (OPP), implying that the system always
converges to ${\mathrm y}^{\star}$. This completes the proof. 
\end{proof}
\begin{rem}
\label{rem:SynthesizeForFormation}If we only assume condition i)
and ii), we get that the system converges to some $\hat{{\mathrm y}}$ which
satisfies $\E^{T}\hat{{\mathrm y}}=\E^{T}{\mathrm y}^{\star}=\zeta^{\star}$. This can
be important in problems in which ${\mathrm y}^{\star}$ is less important than
$\zeta^{\star}$, e.g. when we care about relative outputs (like in
formation control) \cite{LCSS_Paper}.
\end{rem}
\begin{exam}
\label{exa:LinearControllers}Consider the controllers constructed in
$\eqref{eq:LinearControllers}$:
\[
\Pi_{e}:\begin{cases}
\dot{\eta_{e}}=-\eta_{e}+\zeta_{e}-(\xi_{e}+\zeta_{e}^{\star})\\
\mu_{e}=\eta_{e}
\end{cases},
\]
for some $\{\xi_{e}\}_{e\in\EE}$ which are a function of ${\mathrm y}^{\star}$,
and chosen so that condition (1) of Theorem \ref{thm:Synthesis} is satisfied. In that case, we
can compute and see that $\gamma_{e}(\zeta_{e})=\zeta_{e}-\xi_{e}-\zeta_{e}^{\star}$,
so that $\Gamma_{e}(\zeta_{e})=\frac{1}{2}||\zeta_{e}||^{2}-\zeta_{e}^{T}(\xi_{e}+\zeta_{e}^{\star})$
is a strictly convex function, yielding that condition (2) is satisfied. Thus
these controllers always yield the correct relative output $\zeta^{\star}=\E^{T}{\mathrm y}^{\star}$.
\end{exam}
\begin{rem}
Note that the conditions iii) and iv) in Theorem \ref{thm:Synthesis} are controller independent, meaning
that we can always find the correct relative output, but not always
converge to ${\mathrm y}^{\star}$. This is the same phenomenon appearing
in consensus protocols, in which agreement is achieved, but its convergence point
is completely determined by the initial conditions of the agents and
cannot be controlled.
In other words, we can always synthesize for the relative outputs
vector $\zeta^{\star}$, and if ${\mathrm y}^{\star}$ is achievable using synthesis,
the system will converge to it.
\end{rem}

\subsection{Changing the Objective and ``Formation Reconfiguration''}\label{subsec:FormationReconfiguration}

In practical applications, we may want to change the desired output
${\mathrm y}^{\star}$ after some time. However, we wish to avoid a change in
the controller design scheme. Note that in Example \ref{exa:LinearControllers},
we used the desired output ${\mathrm y}^{\star}$ to define the vector $\xi+\zeta^{\star}$. Other than that vector, the controller is independent
of ${\mathrm y}^{\star}$. In \cite{LCSS_Paper}, a partial solution to this problem, named ``Formation Reconfiguration", was introduced for SISO agents and controllers, allowing to solve the synthesis problem for arbitrary desired relative output vector $\zeta^\star\in\mathbb{R}^{|\mathbb{E}|}$ using controller augmentation. In this section, we expand this solution in two manners - we exhibit it for MIMO systems, as well as focus on the synthesis problem for an arbitrary desired output vector $y^\star$.

We wish to implement a similar mechanism for general controllers. We take a stacked controller of the form \eqref{eq:EdgeSystemsODE}, and add a constant exogenous input $\omega=(\alpha,\beta)$,  

\[
\Pi_{\omega}^{\#}:\begin{cases}
\dot{\eta}=\phi(\eta,\zeta-\alpha)\\
\mu=\psi(\eta,\zeta-\alpha)+\beta.
\end{cases}
\]  

This design allows us to alter the design of the system by changing
$\omega=(\alpha,\beta)$, yielding different steady-state outputs. We denote the steady-state output of the closed-loop system with the controller $\Pi_{\omega}^{\#}$ as ${\mathrm y}_0$, i.e., $\Pi_{\omega}^{\#}$ solves the synthesis problem for ${\mathrm y}_0$.

The following result implies that it is enough to solve the synthesis problem for a single output or relative output (e.g., consensus), applying the ``formation reconfiguration" procedure to force any other desired formation.
\begin{thm} [Formation Reconfiguration]
Consider a networked system $(\Sigma, \Pi, \G)$, and suppose that its output converges to ${\mathrm y}_{0}$. Then there
is a function $g:y\mapsto\omega$ such that for any desired achievable
output ${\mathrm y}^{\star}$, satisfying conditions iii) and iv) of Theorem \ref{thm:Synthesis},
if one defines $\alpha=\E^{T}{\mathrm y}^{\star}-\E^{T}{\mathrm y}_{0}$ and $\beta=g({\mathrm y}^{\star})-g({\mathrm y}_{0})$,
then the output of the networked system $(\Sigma, \Pi_\omega^{\#}, \G)$ converges to $y^\star$.
\end{thm}

The controllers produced by the formation reconfiguration scheme are
illustrated in Figure \ref{FormationReconfiguration}.
\begin{proof}
The steady-state input-output relation $\gamma_{\omega}^{\#}$ of $\Pi_{\omega}^{\#}$
can be computed from $\gamma$ using the equation
\[
\gamma_{\omega}^{\#}(\zeta)=\gamma(\zeta-\alpha)+\beta.
\]

Given any achievable ${\mathrm y}$, we know from condition (4) of Theorem $\ref{thm:Synthesis}$
that $k^{-1}({\mathrm y})\cap\IM(\E)\ne\emptyset$, so we take some $\mu_{{\mathrm y}}\in(\R^{d})^{|\EE|}$
such that $-\E\mu_{{\mathrm y}}\in k^{-1}({\mathrm y})$; we define $g({\mathrm y})=\mu_{{\mathrm y}}$.

Now, take some achievable ${\mathrm y}^{\star}$. We denote $\zeta_{0}=\E^{T}{\mathrm y}_{0}$,
and $\zeta^{\star}=\E^{T}{\mathrm y}^{\star}$, so that $\alpha=\zeta^{\star}-\zeta_{0}$,
and $\beta=\mu_{{\mathrm y}^{\star}}-\mu_{{\mathrm y}_{0}}$. Then,
\begin{align*}
k^{-1}({\mathrm y}^{\star})&=k^{-1}({\mathrm y}_{0})+[k^{-1}({\mathrm y}^{\star})-k^{-1}({\mathrm y}_{0})] \\
&\hspace{-20pt}=-\E\gamma(\zeta_{0})-\E(\mu_{{\mathrm y}^{\star}}-\mu_{{\mathrm y}_{0}}) =-\E\left(\gamma(\zeta_{0})-\mu_{{\mathrm y}_{0}}+\mu_{{\mathrm y}^{\star}}\right) \\
&\hspace{-20pt}=-\E(\gamma(\zeta_{0})+\beta)=-\E(\gamma(\zeta_{0})+\beta)=-\E\gamma_{\omega}^{\#}(\zeta_0+\alpha)\\
&\hspace{-20pt}=-\E\gamma_{\omega}^{\#}(\zeta^\star).
\end{align*}
%
%\[
%\begin{aligned}k^{-1}({\mathrm y}^{\star})=k^{-1}(y_{0})+\Bigg[k^{-1}({\mathrm y}^{\star})-k^{-1}(y_{0})\Bigg]=\\=-\E\gamma(\zeta_{0})-\E(\mu_{{\mathrm y}^{\star}}-\mu_{y_{0}})=\\
%=-\E\Big(\gamma(\zeta_{0})-\mu_{y_{0}}+\mu_{{\mathrm y}^{\star}}\Big)=-\E(\gamma(\zeta_{0})+\beta)=\\=-\E\gamma_{\omega}^{\#}(\zeta_0+\alpha)=-\E\gamma_{\omega}^{\#}(\zeta^\star).
%\end{aligned}
%\]
which proves our claim.
\end{proof}

\begin{figure} [t] 
    \centering
    \includegraphics[scale=0.5]{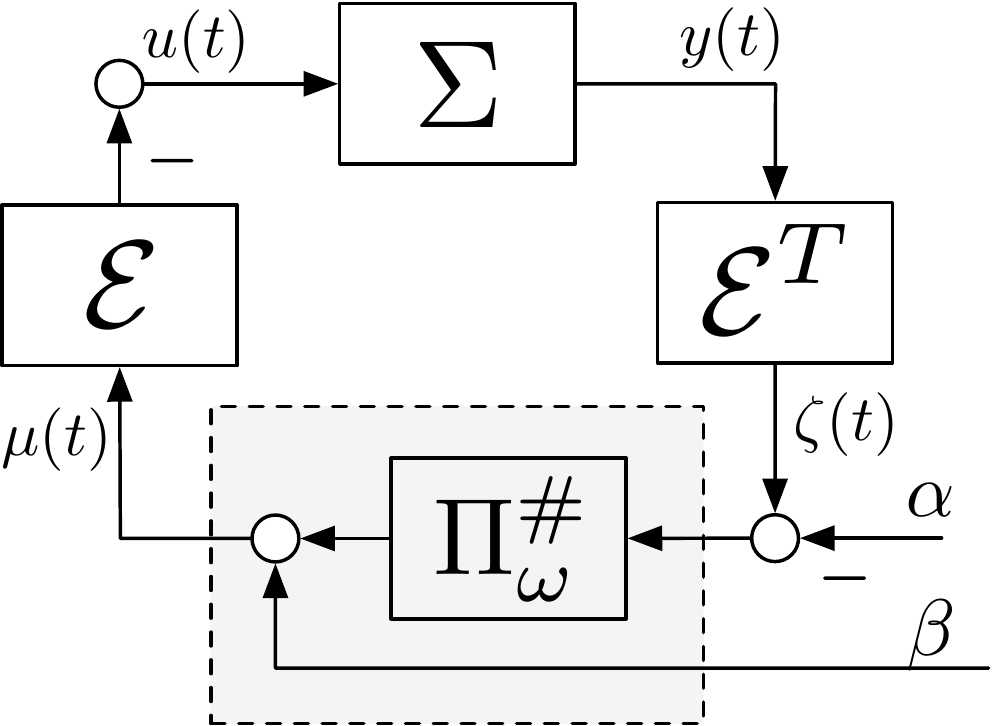}
    \caption{The formation reconfiguration scheme.}
    \label{FormationReconfiguration}
\end{figure} 

\subsection{Plant Augmentation and
Leading Agents for Non-achievable Steady States} \label{subseq:LeadingAgentsSynthesis}

We saw in Section \ref{subsec:CharacterizingForcibleSteadyStates}
that ${\mathrm y}$ can be forced as a steady-state of the system if and only
if ${\bf 0}\in\sum_{i\in\V}k_{i}^{-1}({\mathrm y}_{i})$. This can be troublesome in
applications, in which a certain steady-state can be desired for various
reasons. 

One method of coping with this problem is slightly augmenting the plant. This is done by
introducing a constant external reference signal $\mathrm{z}$ to some of the
nodes. In this direction, we consider a generalized notion of the nodal dynamical systems, 
\begin{equation}
\Sigma^{\prime}_i:\begin{cases}
\dot{x_{i}}=f_{i}(x_{i},u_{i}+\mathrm{z}_{i},\mathrm{w}_i)\\
y_{i}=h_{i}(x_{i},u_{i}+\mathrm{z}_{i},\mathrm{w}_i),
\end{cases}\label{eq:AugmentedPlant}
\end{equation}
%we consider a closed-loop system in the spirit of Figure \ref{ClosedLoop}, in
%which the nodal systems are of the generalized form above. 

Note that if a node is forced to have $\mathrm{z}_{i}=0$, it is of the unaugmented form
we studied earlier. We say that a node is a \emph{follower} if we force it to have $\mathrm{z}_{i}=0$, and we call it a \emph{leader} otherwise. We focus on the case in which there is only one leading node, $i_0\in\V$.
Our interest in leading nodes can be summarized by the following definition.
\begin{defn}
Let ${\mathrm y}\in(\R^{d})^{|\V|}$. We say that the leading node $i_{0}\in\V$
can \emph{force} ${\mathrm y}$ if there is some constant vector $\mathrm{z}_{i_0}$, such that the closed-loop system, with exogenous input $\mathrm{z}_{i_0}$ to the node $i_0$ and zero exogenous output for all nodes $j\ne i_0$, has ${\mathrm y}$ as a steady-state.
We say that the leading node $i_{0}\in\V$ is \emph{omnipotent} if
it can force any vector ${\mathrm y}\in\IM(k)$.
\end{defn}
\begin{thm}
Consider the network system $(\Sigma',\Pi,\mathcal{G})$ and suppose
all agents are MEICMP.  Furthermore let $i_0 \in \V$ be the only leading node (i.e., $\mathrm{z}_i=0$ for all $i\neq i_0$). Then $i_0$ is omnipotent.
\end{thm}
\begin{proof}
Recall that the steady-state input-output relations for the $i$-th node with zero exogenous input were denoted by $k_i$, and denote the steady-state input-output relation for the constant exogenous input $\mathrm{z}_{i_0}$ by $k_{i_0,\mathrm{z}_{i_0}}$. Then
\[
k_{i_0,\mathrm{z}_{i_0}}(u_{i_0})=k_{i_0}(u_{i_0}+\mathrm{z}_{i_0}),\ k_{i_0,\mathrm{z}_{i_0}}^{-1}({\mathrm y}_{i_0})=k_{i_0}^{-1}({\mathrm y}_{i_0})-\mathrm{z}_{i_0}.
\]

Thus, we obtain that $i_{0}\in\V$ can force ${\mathrm y}\in\R^{d}$ if there is some $\mathrm{z}_{i_{0}}\in\R^{d}$ such that
\[
0\in\sum_{i\ne i_0}k_i^{-1}({\mathrm y}_{i})+k_{i_0,\mathrm{z}_{i_0}}^{-1}({\mathrm y}_{i_0})=\sum_{i\in\V}k_{i}^{-1}({\mathrm y}_{i})-\mathrm{z}_{i_{0}}. 
\]
Hence, if we pick $\mathrm{z}_{i_{0}}$ to be some vector in $\sum_{i\in\V}k_{i}^{-1}({\mathrm y}_{i})$, then we get that indeed $0\in\sum_{i\ne i_0}k_i^{-1}({\mathrm y}_{i})+k_{i_0,\mathrm{z}_{i_0}}^{-1}({\mathrm y}_{i_0})$,
allowing to force ${\mathrm y}$ as a steady-state. Thus $i_{0}$ is omnipotent.
\end{proof}
%\begin{rem}
%One can also consider the case of several leading agents, which is
%viable in practical applications in which convergence time is of great
%importance. For example, if the underlying graph has two parts with
%a bottleneck between them, it might be profitable to have two leading
%agents, one at each well-connected part.
%\end{rem}

%\section{Examples of maximally equilibrium independent cyclically monotonic passive systems} 
\section{Examples of MEICMP Systems}\label{sec.ConvGradSystems}
In this section, we focus on giving examples for MEICMP systems, showing that this property holds for many systems found in literature. We focus on two classes of examples, the first being convex-gradient systems with oscillatory terms, generalizing reaction-diffusion systems, gradient descent algorithms and more, and the second being oscillatory systems with damping, which are a natural extension of oscillators like springs and pendulums. We conclude the section with a simulation of a network of oscillatory systems with damping, exemplifying the results of sections \ref{sec.CM_And_Use} and  \ref{sec:Synthesis}. 

\subsection{Convex-Gradient Systems with Oscillatory Terms}

Many systems can be divided into two parts - an oscillatory term and a damping term. These include physical systems such as reaction-diffusion systems, Euler-Lagrange systems and port-Hamiltonian systems, as well as examples coming from optimization theory, in which gradient descent algorithms play a vital role \cite{Port-Hamiltonian_Systems,Port-Hamiltonian_Systems2,Convex_Optimization,Convex_Power_System}. Incremental passivity of these system has been studied in \cite{SymmetricSystems}.  Mathematically, these systems can be represented as
\begin{equation}\label{eq:gradient_system}
\dot{x} = -\nabla\psi(x) + Jx + Bu,
\end{equation}
where $\psi:\mathbb{R}^n\rightarrow\mathbb{R}$ is a function representing the gradient part (and the sign is chosen to give $\psi$ a potential-energy interpretation), $J$ is a skew-symmetric matrix representing the oscillatory part, and $Bu$ is the control input to the system, representing various forces (both control and exogenous ones) acting on the system. Our goal is to show that for a wide class of measurements $y=h(x,u)$, this system is MEICMP. We first focus on stability of this system.

On many occasions, the function $\psi$ is convex, and even strictly convex. For example, $\psi = \frac{\zeta}{2}x^2$ gives a linear damping term.
\begin{thm} \label{thm:ConvexGradient}
Assume that the system (\ref{eq:gradient_system}) is given, and that $\psi$ is a strictly convex function such that for $\underset{||x||\rightarrow \infty}{\lim} \frac{\psi(x)}{||x||}=\infty$. Suppose furthermore that $u$ is constant. Then there exists some \textcolor{black}{unique} $x_0$, which depends on $u$, such that all solutions converge to $x_0$ as $t\to\infty$.
\end{thm}

\textcolor{black}{The proof of the theorem is available in the appendix.} We now deal with the question of cyclic monotonicity. Consider the system
\begin{equation} \label{eq:ConvexGradientSystem}
\begin{cases}
\dot{x}=-\nabla\psi(x)+Jx+Bu\\
y=Cx+\rho(u),
\end{cases}
\end{equation}
where $\psi$ is a strictly convex function such that $\underset{\|x\|\rightarrow \infty}{\lim} \frac{\psi(x)}{\|x\|}=\infty$ and $J$ is a skew-symmetric matrix. By Theorem \ref{thm:ConvexGradient}, the state of the system converges as $t\rightarrow\infty$ whenever $u$ is constant, so the steady-state input-output relation can be defined.
\begin{thm} \label{thm:CMnessOfConvexGradient}
Consider a system of the form \eqref{eq:ConvexGradientSystem}. Suppose that $B$ and $C$ are invertible. Then the input-output relation is CM if the function $(B^{-1}\nabla\psi C^{-1}-B^{-1}JC^{-1})^{-1}+\rho$ is the gradient of a convex function. Furthermore, if this map is the gradient of a strictly convex function, then the input-output relation is SCM.
\end{thm}
\begin{proof}
In steady state, we have $\dot{x}=0$. Thus, if the steady-state input is $\mathrm u_{ss}$ and the state is $\mathrm x_{ss}$, then they relate by $\nabla\psi(\mathrm x_{ss})-J\mathrm x_{ss}=B\mathrm u_{ss}$. As $B$ is invertible, we have
$$
B^{-1}\nabla\psi(\mathrm x_{ss})-B^{-1}J\mathrm x_{ss}=\mathrm u_{ss}.
$$
However, if $\rho=0$ we have $\mathrm y_{ss}^{\rho=0}=C\mathrm x_{ss}$, so we have the relation
$$
B^{-1}\nabla\psi(C^{-1}\mathrm y_{ss}^{\rho=0})-B^{-1}JC^{-1}\mathrm y_{ss}^{\rho=0}=\mathrm u_{ss}.
$$
Thus,
$$
\mathrm y_{ss}^{\rho=0}=(B^{-1}\nabla\psi C^{-1}-B^{-1}JC^{-1})^{-1}(\mathrm u_{ss}).
$$
In the case of general $\rho$, we have the input/output relation
\begin{equation} \label{eq.ssrho}
\mathrm y_{ss}=(B^{-1}\nabla\psi C^{-1}-B^{-1}JC^{-1})^{-1}(\mathrm u_{ss})+\rho(\mathrm u_{ss}).
\end{equation}
\end{proof}
\vspace{-15pt}
\textcolor{black}{
\begin{cor} 
Consider a system of the form \eqref{eq:ConvexGradientSystem}. If $C=B^{T}=I$ and $\rho$ satisfies $(\nabla\psi-J)^{-1}+\rho=\nabla\chi$ for some convex function $\chi$, then the steady-state input-output relation is CM.
\end{cor}
\begin{proof}
This follows directly from \eqref{eq.ssrho} and $C=B^T=I$.
\end{proof}
\begin{cor} 
Consider a system of the form \eqref{eq:ConvexGradientSystem}. If $J=0$ and $\rho(u)$ is the gradient of a convex function, then the steady-state input-output relation is CM.
\end{cor}
\begin{proof}
The only thing that needs to be shown is that $B^{-1}\nabla\psi(C^{-1}u)$ is the gradient of a convex function. Note that this is enough, as the inverse of the gradient function of a convex function is itself the gradient of a convex function (due to duality of convex functions).
To do this, we define $\mu(x)=\psi(C^{-1}x)$. Then $\mu$ is convex
as $\psi$ is, and the gradient of $\mu$ is given by the chain rule.
The $i$-th entry of it is given by
\begin{align*}
\frac{\partial\mu}{\partial x_{i}}&=\sum_{j=1}^{n}\frac{\partial\psi}{\partial x_{j}}(C^{-1}x)\cdot\frac{\partial(C^{-1}x)_{j}}{\partial x_{i}} \\
&=\sum_{j=1}^{n}\frac{\partial\psi}{\partial x_{j}}(C^{-1}x)\cdot(C^{-1})_{ji} =\sum_{j=1}^{n}(C^{-1})_{ji}\frac{\partial\psi}{\partial x_{j}}(C^{-1}x) \\
&=[(C^{-1})^{T}\nabla\psi(C^{-1}x)]_{i}=[B^{-1}\nabla\psi(C^{-1}x)]_{i},
\end{align*}
meaning that $\nabla\mu(x)=B^{-1}\nabla\psi(C^{-1}x)$, proving the
last part.
\end{proof}}

\begin{rem} \label{rem.LinearRelations}
Theorem \ref{thm:CMnessOfConvexGradient}, \textcolor{black}{can be stated more easily for linear systems}.
Suppose that $B,C$ and $J$ are as above. Suppose further that $\psi$ has the form $\psi(x)=x^{T}Ax$ where $A>0$, and suppose we only seek for linear maps $\rho$ of the form $\rho(u)=Tu$ for some matrix $T$. The dynamical system now has the form,
\begin{equation}
\begin{cases}
\dot{x}=-(A-J)x+Bu\\
y=Cx+Tu
\end{cases}.
\end{equation}
We now require $\rho$ to satisfy 
$$
(B^{-1}\nabla\psi C^{-1}-B^{-1}JC^{-1})^{-1}+\rho=\nabla\chi,
$$
for some convex function $\chi$. If we again seek linear $\rho(u)=Tu$, then the left-hand side of the equation is a linear map, so $\nabla\chi$ must also be a linear map. Due to convexity of $\chi$, this is only possible if $\nabla\chi(u)=Du$ for some $D\ge 0$. We end up with following equation,
$
(B^{-1}AC^{-1}-B^{-1}JC^{-1})^{-1}+T\ge0 .
$
After some algebraic manipulation, we obtain
\begin{equation}
C(A-J)^{-1}B+T\ge 0,
\end{equation}
\textcolor{black}{Thus we conclude that a linear system 
\begin{align}
\begin{cases}
\dot{x} = Ax+Bu \\ y = Cx+Tu
\end{cases},
\end{align}
where $A$ is Hurwitz, is MEICMP if and only if $-CA^{-1}B+T$ is a positive-definite symmetric matrix.}
\end{rem}

\subsection{Oscillatory Systems with Damping}

We consider a damped oscillator with a linear forcing term of the form $\ddot{x}+\zeta\dot{x}+\omega^2 x=Bu$
where $B$ is a constant matrix, $u$ is the input vector, $\zeta>0$ is the damping factor. This system can also be represented via the first order ODE:

\begin{equation}
\begin{cases}
\dot{q} = \omega p\\
\dot{p} = -\omega q - \zeta p - Bu \
\end{cases}.
\end{equation}

One can easily generalize this formulation to more complex methods of damping:
\begin{equation}\label{eq:Damping}
\begin{cases}
\dot{q} = M p\\
\dot{p} = -M^T q - \nabla\psi(p) + Bu
\end{cases}.
\end{equation}

We are usually interested in the position as the output, i.e., $y=q$ for this system. We wish to find a condition that will assure this system is stable and MEICMP. We first prove the following result.
\begin{thm}
Consider a system of the form \eqref{eq:Damping}, and suppose that $M$ is invertible. Suppose furthermore that $\psi$ is a strictly convex function such that $\underset{\|x\|\rightarrow \infty}{\lim} \frac{\psi(x)}{\|x\|}=\infty$. Then the system  is stable for constant inputs.
Furthermore, if the system is injected with the constant input signal $\mathrm u$, then there is some $q_0$ such that all trajectories of the system satisfy $q\rightarrow q_0 , p\rightarrow p_0 = 0$ as $t\rightarrow\infty$. Even further, $q_0 = (M^T)^{-1}Bu-(M^T)^{-1}\nabla\psi(p_0)$
\end{thm}

\begin{proof}
As above, the assumption on $\psi$ allows us to absorb the linear term inside $\psi$, so we can assume $B\mathrm u =0$. Now, we take $p_0 = 0$ and $q_0 = -(M^T)^{-1} \cdot \nabla\psi(p_0)$. We note that the following relations hold:
\begin{align}\label{eq:char_equations_damping}
Mp_0 &= 0 ,\; M^T q_0 = -\nabla\psi(p_0),\; p_0^T\nabla\psi(p_0) = 0.
\end{align}
Now, consider the following Lyapunov function candidate,
\begin{equation}
F(p,q)=\frac{1}{2}(p-p_0)^T(p-p_0) +\frac{1}{2}(q-q_0)^T(q-q_0).
\end{equation}
It's clear that $F\ge 0$ and that $F=0$ if and only if $p=p_0$ and $q=q_0$. Furthermore, the derivative of $F$ along the trajectories is given by:
\begin{align*}
\dot{F}&=(p-p_0)^T\dot{p}+(q-q_0)^T\dot{q} \\
&= (p-p_0)^T(-M^T q - \nabla\psi(p))+ (q-q_0)^TMp \\
&= -(p-p_0)^T\nabla\psi(p) - (Mp_o)^Tq - (M^Tq_0)^Tp \\
& \overset{\eqref{eq:char_equations_damping}}{=} -(p-p_0)^T\nabla\psi(p) + p^T\nabla\psi(p_0)-p_0^T\nabla\psi(p_0) \\
&= -(p-p_0)^T(\nabla\psi(p)-\nabla\psi(p_0)).
\end{align*}
The last expression is non-positive, and furthermore is strictly negative if $p\ne p_0$ (as $\psi$ is strictly convex). Thus, it's clear that $p\rightarrow p_0=0$ as $t\rightarrow \infty$. Now, the equation driving $p$ is $\dot{p} = -M^T q - \nabla\psi(p)$, Which can be rewritten as
\begin{equation}
q=-(M^T)^{-1} (\dot{p}+\nabla\psi(p)).
\end{equation}
When the time grows infinite, the right hand side tends to $-(M^T)^{-1}(\nabla\psi(p_0))=q_0$, concluding the proof of the claim.
\end{proof}

Not only have we proved that the system is stable, we also found the input-output steady-state relation, which turns out to be linear. Thus, we can apply Remark \ref{rem.LinearRelations} to conclude the following corollary.  

\begin{cor}
The system (\ref{eq:Damping}) is MEICMP if and only if the matrix $(M^T)^{-1}B$ is positive semi-definite.
Furthermore, it is MEISCMP if and only if this matrix is positive definite.
\end{cor}

We now demonstrate these results for oscillatory systems with damping by a simulation.
\begin{exam}
We consider a network of four damped MIMO oscillators,
\begin{align} \nonumber
\Sigma_i\begin{cases}
\begin{bmatrix} \dot{x}_1 \\ \dot{x}_2 \end{bmatrix} = \begin{bmatrix} \Omega_i x_2 \\ -D_ix_2 -\Omega_i^T (x_1-\mathrm{x}_i) + \Omega_i^{-1} u \end{bmatrix} \\ y = x_1
\end{cases}
\end{align} 
where $\mathrm{x}_i$ is the equilibrium point of the oscillator, $\Omega_i$ is a matrix consisting of the self frequencies, and $D_i$ is a damping matrix, which is positive-definite. The exact values of the matrices and $x_i$-s were randomly chosen. The underlying graph is given in Figure \ref{SimulationGraph}.

The steady-state input-output relation if $\Sigma_i$ can be computed to be $k_i(u_i)=(\Omega_i\Omega_i^T)^{-1}u+\mathrm{x}_i$, whose inverse is $k_i^{-1}(y)=(\Omega_i\Omega_i^T)(y-\mathrm{x_i})$. This gives us the convex function $K_i^\star(y_i)=\frac{1}{2}y^T\Omega_i\Omega_i^Ty-y^T\Omega_i\Omega_i^T\mathrm{x}_i$, which is strictly convex.

We wish to solve the synthesis problem for $y^\star$, where the controllers are taken to be identical and equal to
\[
\begin{cases}
\dot{\eta_e} = -\eta_e+\zeta_e \\
\zeta_e = \psi(\eta_e).
\end{cases}
\]
The function $\psi$ is given as
\[
\psi(x) = \arcsin\bigg(\frac{\log^2\big(\frac{e^x+1}{2}\big)\mathrm{sgn}(x)}{\log^2\big(\frac{e^x+1}{2}\big)+1}\bigg),
\]
where $\mathrm{sgn}(x)$ is the sign function. \textcolor{black}{One can verify that $\psi(0)=0$ and that $\psi$ is a monotone ascending function. The associated integral function is given by:
\[
\Gamma_e (\zeta_e) = \int_0^{\zeta_e} \arcsin\bigg(\frac{\log^2\big(\frac{e^x+1}{2}\big)\mathrm{sgn}(x)}{\log^2\big(\frac{e^x+1}{2}\big)+1}\bigg)dx.
\]}

\begin{figure} [t] 
    \centering
    \includegraphics[scale=0.3]{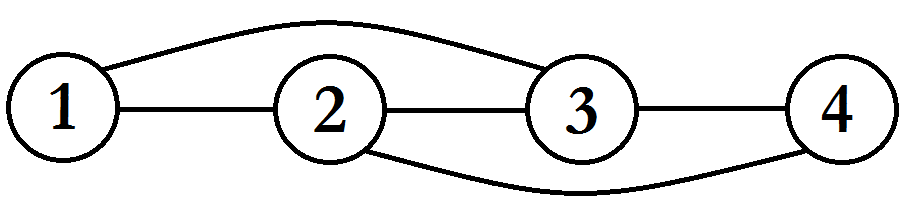}
    \caption{The graph used for the example.}
    \label{SimulationGraph}
\end{figure}

We then use the formation reconfiguration scheme to create an augmented controller, where we use the first node as a leading node. The control objective was changed every 30 seconds according to the following desired steady-states,
\small
\begin{align}\nonumber
y^{\star 1}&=[0,0,0,0,0,0,0,0]^T, & y^{\star 2}&=[1,1,2,2,3,3,4,4]^T, \\\nonumber
y^{\star 3}&=[1,2,3,4,5,6,7,8]^T, & y^{\star 4}&=[-1,0,0,0,1,0,2,2]^T, \\\nonumber
y^{\star 5}&=[2,2,2,2,2,2,-10,-10]^T,& &
\end{align}
\normalsize
where the first two entries refer to the first agent, the next two refer to the second agent, and so on. The output of the system can be seen in Figure \ref{UnifiedControlSchemePositionsFig}, exhibiting the positions of the agents $y(t)$. The blue line represents first coordinate, and the red one represents the second coordinate. We can see that the agents act as expected, converging to the desired formations.
\end{exam}
\begin{figure} [!t] 
    \centering
    \includegraphics[scale=0.7]{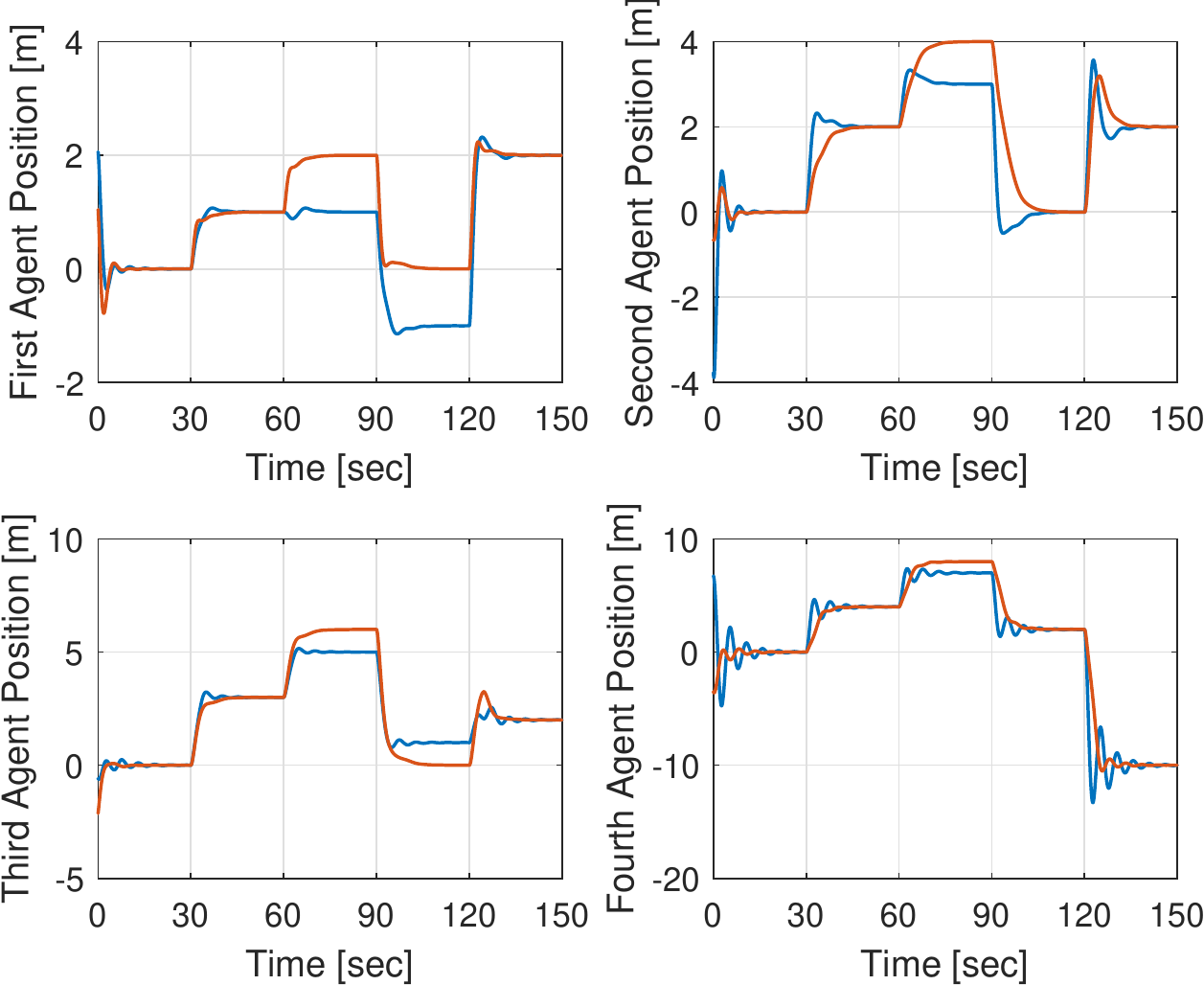}
    \caption{Formation control of damped MIMO oscillators.}
    \label{UnifiedControlSchemePositionsFig}
    \vspace{-15pt}
\end{figure}

\section{Conclusion}

We have found a profound connection between passivity-based cooperative control and network optimization theory in the spirit of Rockafellar \cite{Rockafeller}. This was done by introducing the notion of maximal equilibrium-independent cyclically monotone passive systems, and showing that such systems converge to a solution of a collection of network optimization problems, bonded by duality.
Furthermore, we have shown that in the case of output-agreement problems, the output agreement steady-state is optimal with respect to the optimal flow and optimal potential problems. This connection creates a dictionary between system signals (like outputs and inputs) and network optimization variables (potentials and node divergences, respectively).
We have established analogous inverse optimality and duality results for general networks of maximal equilibrium-independent cyclically monotone passive systems.
Using methods from subgradient theory and convex optimization, we have established clear criteria for solvability of the synthesis problem for a diffusive coupling of maximal equilibrium-independent cyclically monotone passive systems, and a practically-justifiable plant augmentation procedure to solve the synthesis problem if it is not feasible for the desired output. We have shown a synthesis for the controllers, and exhibited a controller augmentation procedure that allows any set of maximally equilibrium-independent cyclically monotone passive controllers to be used. This theory was exemplified by simulating a system of damped planar oscillators and correctly predicted the asymptotic state of the system, using both the minor controller and plant augmentation procedures. 
We believe that this strong connection between passivity-based cooperative control and network optimization theory can lead to new analysis methods for cooperative control problems, through the means of network optimization problems.

\textcolor{black}{This is a significant extension of the framework connecting multi-agent systems and cooperative control to network optimization, first presented in \cite{SISO_Paper} and later developed in \cite{LCSS_Paper}. Possible further research directions can include extensions of the framework (e.g., to directed graphs, passivity-short systems, and systems with different input and output dimension), or applications of the framework to yield various results in multi-agent systems (e.g. fault detection and isolation, network identification, robustness, etc.). }%Lastly, we look to apply our results to system of a specific form, e.g. neural networks, in order to obtain convergence results for said networks.}
\if(0)
%\medskip
\appendix
\section{Proving That CM Relations Generalize Monotone Relations}
In theorem \ref{thm.CM_Generalizes_Monotone}, we showed that in one dimension, CM relations are the same as monotone relations. However, we needed two lemmas for the proof, and they were cited above. For the sake of completeness, the proofs of the lemmas are given below. As for the first lemma:
\begin{lem}
Let $f:I\to\mathbb{R}$ be a monotone increasing function defined on an interval $I$. Fix some $x_0\in I$ and define $g:I\to\mathbb{R}$ by $g(x)=\int_{x_0}^{x}{f(t)dt}$. Then $g$ is a convex function, and the sub-differential $\partial g$ contains the set $\{(x,f(x): x\in I\}$. Furthermore, if $f$ is strictly ascending, then $g$ is strictly convex
\end{lem}
\begin{proof}
We first take some arbitrary $x,y\in I$ and some $\alpha\in [0,1]$. Assume without loss of generalization that $x<y$. Note that changing the point ${x_0}$ changes the value of $g$ by a constant, unharming all of the wanted properties. Thus we are free to choose $x_0$ as we want when trying to prove the claim. Choose $x_0 = x$. Then because $f$ is monotonic ascending, if we denote $z = \alpha x + (1-\alpha)y$ then:
\small
\begin{align}
& \alpha g(x) + (1-\alpha) g(y) = \alpha\int_{x}^{x}{f(t)dt} + (1-\alpha) \int_{x}^{y} {f(t)dt} \nonumber\\
= &(1-\alpha)\int_{x}^{z} {f(t)dt} + (1-\alpha) \int_{z}^{y}{f(t)dt} \nonumber\\
\ge &(1-\alpha)\int_{x}^{z} {f(t)dt} + (1-\alpha) (y-z)f(z) \nonumber\\
= &(1-\alpha)\int_{x}^{z} {f(t)dt} + (1-\alpha)\alpha (y-x) f(z) \nonumber\\
= &(1-\alpha)\int_{x}^{z} {f(t)dt} + \alpha (z - x) f(z) \nonumber\\
\ge &(1-\alpha)\int_{x}^{z} {f(t)dt} + \alpha \int_{x}^{z} {f(t)dt} = \int_{x}^{z} {f(t)dt} = g(z)\nonumber .
\end{align}
\normalsize
this shows that $g$ is convex. Note that if $f$ was strictly monotonic ascending, then the inequalities are strict if $x\ne y$.
We now take some $x_1\in I$ and claim that $(x_1,f(x_1))$ is contained in the sub-differential of $g$. Indeed, we need to take some general $x\in\mathbb{R}$ and show that $g(x) - g(x_1) \ge f(x_1)(x-x_1)$. This can be read as $\int_{x_1}^{x} {f(t)dt} \ge f(x_1)(x-x_1)$. We consider two cases:
\begin{itemize}
\item[i)]  $x \ge x_1$: In this case the minimum value of the function inside the integral is $f(x_1)$ (as $f$ is monotonic ascending), and the inequality is clear;
\item[ii)] $x \le x_1$: In this case, the integral is equal to $-\int_{x}^{x_1} {f(t)dt}$, which can be bounded from below by $-f(x_1)(x_1-x)$ because $f$ is monotonic ascending. 
\end{itemize}
This completes the proof of the lemma.
\end{proof}

And as for the second lemma:
\begin{lem}
Let $R$ be a monotone relation. Then for all but countably many $u\in \mathbb{R}$, there is at most one $y\in \mathbb{R}$ such that $(u,y)\in R$.
\end{lem}

\begin{proof}
We'll show that for any $u_{-},u_{+}$ and any $\epsilon>0$, there are only finitely many $u$-s between them such that there are $y_0,y_1$ such that $y_1 - y_0 > \epsilon$. Fix some corresponding $y_{+}$ and $y_{-}$.
It's clear that if $u$ has the form above, and $u_a>u>u_b$ and $(u_a,y_a),(u_b,y_b)\in R$ then $y_b-y_a > \epsilon$ (as $R$ is monotonic).
Thus, if there are $N$ points like $u$ above in the segment between $u_{+}$ and $u_{-}$ then we get that $f(y_{+}) - f(y_{-}) \ge N\epsilon$. Thus $N$ is a finite number, and discretely decreasing $\epsilon$ and expending the range completes the proof.
\end{proof}
\fi
\bibliographystyle{ieeetr}
\bibliography{main}

\appendix
This appendix deals with the proof of Theorem \ref{thm:ConvexGradient}. The proof is rather lengthy, and requires two lemmas The idea is to try and construct a quadratic Lyapunov function of the form $V(x)=\frac{1}{2}(x-x_0)^T(x-x_0)$, where the point $x_0$ is a fixed point of the flow. Thus, we need to find a point $x_0$ which satisfies $\nabla\psi(x_0)-Bu = Jx_0$. The following two lemmas will assure that such a point exists. % in our case. 

%The first lemma is:
\begin{lem} \label{lem.pointsOutward}
Let $\chi$ be a strictly convex function, and suppose that $\frac{\chi(x)}{||x||}\to\infty$ as $||x||\to\infty$. Then there exists some $\rho>0$, such that for every point $x\in\mathbb{R}$ satisfying $\|x\|=\rho$, the inequality $\langle x, \nabla\chi(x)\rangle\ge 0$ holds.
\end{lem}

\begin{proof}
Fix some arbitrary unit vector $\theta\in\mathbb{R}^n$, and consider the convex function $f_\theta (r) = \chi(r\theta)$ and its derivative $\frac{df_\theta}{dr}=\nabla\chi(r\theta)^T\theta$. Note that because $\chi$ grows faster than any linear function, the same can be said about $f_\theta$, and in particular, it's derivative tends to infinity.
Furthermore, the function $f_\theta$ is strictly convex, so $\frac{df_\theta}{dr}$ is strictly ascending, Thus there is some $r_\theta$ such that $\frac{df_\theta}{dr}>0$  if $r>r_{\theta}$ and $\frac{df_\theta}{dr}<0$  if $r<r_{\theta}$.

Our task now is to show that $r_\theta$ is a bounded function of $\theta$. Suppose not, and let $\theta_n$ be a sequence of unit vectors such that $r_{\theta_n}\to\infty$. Passing to a subsequence, we may assume without loss of generality that $\theta_n\to\theta$ for some unit vector $\theta\in\mathbb{R}^n$. There is some $N$ such that if $n\ge N$ then $r_{\theta_n}>r_\theta+1=t$. In particular, $\frac{df_{\theta_n}}{dr}|_{r=t}\le 0$ for $n\ge N$ but $\frac{df_{\theta}}{dr}|_{r=t} > 0$. This is impossible, as the first expression is equal to $\nabla\chi(t\theta_n)^T\theta_n$, which converges to the second expression, which is $\nabla\chi(t\theta)^T\theta$.  Thus, there is some $\rho>0$ such that $r_\theta<\rho$  for all unit vectors $\theta$, meaning that if $x$ is a vector of norm $\rho$, then for $\theta=\frac{x}{||x||}$:
\begin{equation}
\langle \nabla\chi(x),x \rangle = \rho \langle \nabla \chi(\rho\theta),\theta \rangle =\rho\frac{df_\theta}{dr}(\rho) \ge 0.
\end{equation}
\end{proof}

%Now, for the second lemma.
\begin{lem}\label{lem.qy}
Let $Q:\mathbb{R}^n\rightarrow\mathbb{R}^n$ be a continuous vector field, and let $\rho>0$. Suppose that for any vector $x$ satisfying $\|x\|=\rho$, the inequality $\langle Q(x),x\rangle \ge 0$ holds. Then there exists some point $y$ satisfying $\|y\|\le \rho$ such that $Q(y)=0$.
\end{lem}

In order to prove the lemma, we use a theorem from algebraic topology.
\begin{thm}[Brouwer's Fixed Point Theorem \cite{Algebraic_Topology}]
Let $D$ be a closed ball inside $\mathbb{R}^n$, and let $f:D\rightarrow D$ be a continuous map. Then $f$ has a fixed point.
\end{thm}

Now, we prove the lemma.
\begin{proof}
Suppose, heading toward contradiction, that $Q$ does not vanish at any point in the ball $D=\{\|x\|\le \rho\}$. We define a map $F:D\to D$ by
\begin{equation}
F(x) = -\rho\frac{Q(x)}{\|Q(x)\|}.
\end{equation}

This is a continuous map (as $Q$ never vanishes), and the norm of $F(x)$ is always equal to $rho$, so $F(x)$ is indeed in $D$. Thus, we can apply Brouwer's fixed point theorem to $F$ and get a fixed point, called $y$.

We know that $y$ satisfies $F(y)=y$, i.e., $-\rho\frac{Q(y)}{\|Q(y)\|} = y$. On one hand, taking the norm of the last equation implies that $\|y\|=\rho$. On the other hand, rearranging it implies that $Q(y)=-\frac{\|Q(y)\|}{\rho}y=\lambda y$ where $\lambda$ is some negative scalar (as $Q(y)\neq 0$). Thus, we found a point $y$ of norm $\rho$ such that $\langle Q(y),y\rangle = \lambda\|y\|^2{<}0$ 
for some $\lambda < 0$, which contradicts our assumption. Thus $Q$ has a zero inside the ball $D=\{||x||<\rho\}$.
\end{proof}

We are now ready to prove Theorem \ref{thm:ConvexGradient}.
\begin{proof}
First, because $u$ is constant, we can absorb the constant term $Bu$ inside the gradient $\nabla\psi(x)$ by adding the linear term $(Bu)^Tx$ to $\psi(x)$. This does not change the fact that $\psi$ is strictly convex, nor the fact that it ascends faster than any linear function. Thus we may assume that $Bu=0$ for the remainder of the proof.

Now, we define the vector field $Q(x)=\nabla\psi(x)-Jx$. Note that because $J$ is skew-symmetric, for all $x\in\mathbb{R}$,
\begin{equation}
\langle \nabla\psi(x) - Jx,x\rangle = \langle\nabla\psi(x),x\rangle.
\end{equation}
Thus, by the Lemma \ref{lem.pointsOutward}, there's some $\rho>0$ such that $\langle Q(x),x\rangle \ge 0$ for any vector $x$ satisfying $\|x\|\le \rho$, and by Lemma \ref{lem.qy} we can find some point $x_0\in\mathbb{R}$ such that $Q(x_0)=0$, or equivalently, $Jx_0 = \nabla\psi(x_0)$. We claim that any solution to the ODE converges to $x_0$.

Indeed, define $F(x) = \frac{1}{2} \|x-x_0\|^2$. Then $F$ is non-negative, and vanishes only at $x_0$, and furthermore,
%\begin{equation}
%\begin{split}
%&\dot{F} = (x-x_0)^T\dot{x} = (x-x_0)^T(-\nabla\psi(x)+Jx)\overset{Jx_0=\nabla\psi(x_0)}{=} \\
%&(x-x_0)^T(-\nabla\psi(x)+\nabla\psi(x_0)+J(x-x_0))\overset{J=J^T}{=} \\
%&-(x-x_0)^T(\nabla\psi(x)-\nabla\psi(x_0)).
%\end{split}
%\end{equation}
\begin{align}
\dot{F} &=(x-x_0)^T\dot{x} = (x-x_0)^T(-\nabla\psi(x)+Jx)\nonumber \\
&=(x-x_0)^T(-\nabla\psi(x)+\nabla\psi(x_0)+J(x-x_0))\nonumber \\
&=-(x-x_0)^T(\nabla\psi(x)-\nabla\psi(x_0))\le 0,
\end{align}
%\begin{equation}
%\begin{split}
%&\dot{F} = (x-x_0)^T\dot{x} = (x-x_0)^T(-\nabla\psi(x)+Jx){=} \\
%&(x-x_0)^T(-\nabla\psi(x)+\nabla\psi(x_0)+J(x-x_0)){=} \\
%&-(x-x_0)^T(\nabla\psi(x)-\nabla\psi(x_0))\le 0.
%\end{split}
%\end{equation}
where the last inequality is true because $\psi$ is convex and Theorem \ref{thm:Rockafellars}. Furthermore, $\dot{F}$ is negative if $x\neq x_0$ because $\psi$ is strictly convex and Theorem \ref{thm:Rockafellars}. \textcolor{black}{The uniqueness of $x_0$ follows from the fact that the flow globally asymptotically converges to $x_0$}. This completes the proof.
\end{proof}

\if(0)
\begin{IEEEbiography}
 {\bf Daniel~Zelazo} is an Assistant Professor of Aerospace Engineering at the Technion - Israel Institute of Technology. He received his BSc. (99) and M.Eng (01) degrees in Electrical Engineering from the Massachusetts Institute of Technology. In 2009, he completed his Ph.D. from the University of Washington in Aeronautics and Astronautics. From 2010-2012 he served as a post-doctoral research associate and lecturer at the Institute for Systems Theory \& Automatic Control in the University of Stuttgart. His research interests include topics related to multi-agent systems, optimization, and graph theory.
 \end{IEEEbiography}

\begin{IEEEbiography}
 {\bf Miel~Sharf} is a Ph.D. student in the Aerospace Engineering department at the Technion - Israel Institute of Technology. He received his B.Sc. (13) and M.Sc. (16) degrees in Mathematics from the Technion - Israel Institute of Technology. His research interests include topics related to multi-agent systems.
 \end{IEEEbiography}
\fi

\end{document}